\documentclass[11pt,reqno]{amsart}%
\usepackage[a4paper,margin=3cm]{geometry}%
\usepackage[T1]{fontenc}%
\usepackage[utf8]{inputenc}%
\usepackage{lmodern,mathtools}
\usepackage{amsmath,amsfonts,amssymb,amsthm}
\usepackage{amsthm}
\usepackage{tikz}
\usepackage{dsfont}

\RequirePackage[colorlinks,linkcolor=blue,citecolor=blue,urlcolor=blue]{hyperref}
\title[]{Large deviations for mean field particle systems via analysis on spaces of measures}
\author{Valentin Pesce}
\address[VP]{CMAP, Ecole Polytechnique, UMR 7641, 91120 Palaiseau, France.}
\email{\url{mailto:valentin.pesce@polytechnique.edu}}

\date{Summer 2026, compiled \today}
\theoremstyle{plain}
\newtheorem{theorem}{Theorem}[section]%
\newtheorem{definition}[theorem]{Definition}
\newtheorem{corollary}[theorem]{Corollary}%
\newtheorem{proposition}[theorem]{Proposition}%
\newtheorem{lemma}[theorem]{Lemma}%
\newtheorem{remark}[theorem]{Remark}%
\newtheorem{claim}{Claim}%
\let\ensembleNombre\mathbb
\newcommand\tab[1][0.5cm]{\hspace*{#1}}
\newcommand\mes{\mathcal P(\R^d)}
\newcommand\mesT{\mathcal P(\mathbb T^d)}

\newcommand\mesTdeux{\mathcal P_2(\T^d)}
\newcommand\N{\ensembleNombre{N}}
\newcommand\Z{\ensembleNombre{Z}}

\newcommand\R{\ensembleNombre{R}}

\newcommand\T{\ensembleNombre{T}}
\newcommand\dis{\displaystyle}
\newcommand\E{\mathbb{E}}
\newcommand\mP{\mathbb{P}}

\DeclareMathOperator{\leb}{Leb} 
 
\makeatletter
\def\@MRExtract#1 #2!{#1}     
\renewcommand{\MR}[1]{
  \xdef\@MRSTRIP{\@MRExtract#1 !}%
  \href{http://www.ams.org/mathscinet-getitem?mr=\@MRSTRIP}{MR-\@MRSTRIP}}
\makeatother

\numberwithin{equation}{section}


\begin{document}

\begin{abstract}
This paper introduces a robust strategy to establish large deviation principles for mean field interacting particle systems in $\T^d$ with vanishing noise. Viewed as a Freidlin-Wentzell type problem, our approach relies on recent techniques for partial differential equations (in short PDE) on the space of probability measures \cite{bertucci2026analysis}. We first outline the proof strategy for the PDE approach to large deviations in a self-contained manner. As the number of particles goes to infinity, we then show that the transfer function of the system converges to the unique viscosity solution of a limiting Hamilton-Jacobi equation on the space of probability measures. Finally, we explicitly derive the rate function by formulating the solution of this limiting PDE as the value function of an optimal control problem.
\end{abstract}

\maketitle
 
{\footnotesize\tableofcontents}
\section{Introduction}

\subsection{Motivation and results}
The goal of this paper is to establish a Freidlin-Wentzell \cite{freidlin1998random} large deviations type result for systems of particles in mean field interaction using some techniques of analysis on the space of probability measures. We shall consider systems of particles in mean field interaction with a stochastic perturbation that vanishes when the number of particles increases. 
\newline
Such systems of particles appeared in various topics as the propagation of chaos \cite{sznitman2006topics,meleard1987propagation,chaintron2022propagation} or random matrices \cite{dyson1962,anderson2010,bru1991wishart,demni2009radial,guskov2026dyson}. We shall study in this paper a toy model to present a strategy for proving such large deviations results using the viscosity solutions theory for Hamilton-Jacobi equations on the space of probability measures. 
\newline
Mostly during the eighties, Fleming et al. \cite{fleming1977exit,fleming1986pde,barles1988exit} studied some problems of large deviations for random, continuous in time, processes with values in $\R^d$ using the theory of viscosity solutions \cite{crandall1992}. A natural connection between large deviations and viscosity solutions theory can be done since the rate functions that measure the deviations of these processes can be characterized as the infimum of certain energies over admissible trajectories \cite{freidlin1998random,boue1998variational,varadhan1966asymptotic}.
This infimum can also be viewed as the value functions of optimal control problems which are known to be the unique viscosity solution of Hamilton-Jacobi equations \cite{crandall1983viscosity,crandall1984some,fleming1977exit}.
\newline
Let $d\ge 1$ be the dimension, $b:\T^d\times \mesTdeux\to\R^d$ be a Lipschitz drift term , $\sigma_N> 0$ such that $\sigma_N\to 0$ when $N\to+\infty$ and a horizon of time $T>0$. Consider the system of $N\ge 1$ particles on $\T^d$ in mean field interaction 
\begin{equation}
\label{eq:systemofparticlesintro}
dx_t^i=b(x_t^i,\mu_t^N)dt+\sigma_N dB_t^i \mathrm{ \,\,modulo \,\,}\Z^d, \, 1\le i\le N
\end{equation}
where $(B^i)_{i\in\N}$ are independent and identically distributed (in short iid) standard Brownian motions in $\R^d$ defined on a same probability space and $\mu_t^N=N^{-1}\sum_{1\le i\le N}\delta_{x_t^i}$ is the empirical distribution of the particles.
\newline
The goal is to prove a large deviations principle of the sequence of random variables $(X_N)_{N\ge 1}$ defined by 
\begin{equation}
\label{eq:randomvariablesintro}
\begin{split}
X_N \,:\, (\Omega,\mathcal F ,\mP)&\to C([0,T],\mesTdeux)\\
\omega&\mapsto \left(t\mapsto \frac{1}{N}\sum_{i=1}^N\delta_{x_t^i(\omega)}\right),
\end{split}
\end{equation}
where $(x_t^i)_{1\le i\le N}$ is a strong solution of \eqref{eq:systemofparticlesintro}. 
\newline
\tab Our strategy shall follow the ideas of the work of Fleming et al. Here, the main difficulty in comparison with their work is that at the limit, the partial differential equation (in short PDE) is set on the space of probability measures. The understanding of PDEs on the space of probability measures is a quite recent and active topic \cite{bertucci2026analysis} which was mostly initiated by the study of the mean field games introduced by Lasry and Lions \cite{lasry2007mean,lions2007-2011cours}. Since, mean field games have become an important study of interest and have investigated the notion of derivatives and PDEs on the space of probability measures \cite{lions2007-2011cours,cardaliaguet2010notes,cardaliaguet2015master,bertucci2026analysis}.
\newline
Formally, the key quantity of interest to prove the large deviations of $(X_N)_{N\ge 1}$ will be
\begin{equation}
\label{eq:ratefunctionlaplace}
\Lambda_N^f(t,x)=\cfrac{\sigma_N^2}{N}\,\log\left(\E_x\left[\exp(N\sigma_N^{-2}f(X_N(t)))\right]\right),\, (t,x)\in\R^+\times (\T^d)^N
\end{equation} for every $f\in BL(\mesTdeux)$ (where $BL(\mesTdeux)$ is the set of bounded and Lipschitz functions from $\mesTdeux$ to $\R$), which is the Cole-Hopf transform of the solution of a linear parabolic partial differential equation. This quantity is the strong solution to a nonlinear PDE on $(\T^d)^N$. When $N$ goes to infinity, we shall show that its limit $\mathcal U^f$ is the unique viscosity solution of the following Hamilton-Jacobi equation on the space of probability measures 
\begin{equation}
\label{eq:pdeintro}
\begin{split}
\partial_t \mathcal U(t,\mu)=\dis\int_{\T^d}b(y,\mu).D_\mu\mathcal U(t,\mu,y)\mu(dy)+\cfrac{1}{2}\dis\int_{\T^d}|D_{\mu}\mathcal U(t,\mu,y)|^2\mu(dy),\, \, t>0,\, \mu\in\mesTdeux .
\end{split}
\end{equation}
The uniform convergence of $\Lambda_N^f$ toward $\mathcal U^f$ is linked with some active questions that investigate how to approximate viscosity solutions of Hamilton-Jacobi equation on infinite dimensional space by finite dimensional projections of this solution \cite{cardaliaguet2026mean,cardaliaguet2023algebraic,gangbo2021finite,swikech2025finite,fabbri2017stochastic}.
\newline
\tab More precisely the main result that we are going to prove is the following one. 
\begin{theorem}
\label{thm:largedev1}
Assume that the family $(\mathrm{ Law}(X_N(0)))_{N\ge 1}$ satisfies a large deviations principle in $\mesTdeux$ at speed $\varepsilon_N:=N/\sigma_N^2$ with transfer function $\Lambda_0$ defined by $$\Lambda_0(g)= \underset{N\to +\infty}{\lim}\frac{1}{\varepsilon_N}\log\left(\E\left[\exp\left(\varepsilon_N g(X_N(0))\right)\right]\right),$$ for $g\in BL(\mesTdeux)$. 
\newline
Then the sequence of random variables $(X_N)_{N\ge1}$ satisfies a large deviation principle in $C([0,T],\mesTdeux)$ at speed $N/\sigma_N^2$ with good rate function $I$. More precisely, $I$ is given by $$I(\mu)=\sup_{m\ge 1}\sup_{0\le t_1\le ...\le t_m\le T} I_{t_1,...,t_m}(\mu(t_1),...,\mu(t_m))$$ 
where $$I_{t_1,...,t_m}(\mu_1,...,\mu_m)=\sup_{f_1,...,f_m\in BL(\mesTdeux)}\left[\sum_{i=1}^m f_i(\mu_i)-\Lambda_0(\mathcal U^{f_1+\mathcal U^{f_2+...}(t_2-t_1)(\bullet)}(t_1,\bullet))\right],$$
where for all $h\in BL(\mesTdeux)$, $\mathcal U^h:[0,T]\times \mesTdeux\to\R$ is the unique viscosity solution of \eqref{eq:pdeintro} with initial condition $h$.
\newline
Moreover, the good rate function $I$ can be written as follows: $$I(\mu)=I_0(\mu(0))+\cfrac{1}{2}\inf_{\substack{(v_s)_{s\in[0,T]};\\\partial_s \mu_s+\nabla.(\mu_s v_s)=0}}\int_0^T\int_{\T^d}|v_s(x)-b(x,\mu_s)|^2\mu_s(dx)ds$$ if $\mu:[0,T]\to \mesTdeux$ is absolutely continuous and $I(\mu)=+\infty$ otherwise.
\end{theorem}
\subsection{Structure of the paper}
In the first part of the paper, we recall general principles in the theory of large deviations that motivate our strategy for particle systems. Despite being mostly known by specialists some of these results are often presented as applications of some abstract results in large deviations theory. We give some elementary proofs of these results in our context. Typically, we recall that to prove that $(X_N)_{N\ge 1}$ satisfies a large deviations principle it suffices to show that the sequence is exponentially tight and has a unique admissible good rate function. We then prove that finding the good rate function for $C([0,T],\mathbb X)$ valued random variables $(Y_N)_{N\ge 1}$ reduces to studying its finite-dimensional marginals. We finally give a criterion to characterize the good rate function of random variables with values in product spaces using the Varadhan principle.
\newline
In the second part of this paper, we establish Theorem \ref{thm:largedev1}. We first show that the sequence $(X_N)_{N\ge1}$ defined in \eqref{eq:randomvariablesintro} is exponentially tight at speed $N/\sigma_N^2$. Then, using the Markov property we explain why the study of the good rate function of the finite dimensional marginals leads to the study of the limit of $\Lambda_N^f$ defined in \eqref{eq:ratefunctionlaplace} for $f\in BL(\mesTdeux)$. This function is the Cole-Hopf transform of the solution of a linear parabolic PDE and so is the solution of a Hamilton-Jacobi equation on $\R^+\times(\T^d)^N$. The study of the limit formally leads to the Hamilton-Jacobi PDE \eqref{eq:pdeintro} set on the space of probability measures. Using the formalism introduced in \cite{bertucci2026analysis}, we define a notion of viscosity solution of \eqref{eq:pdeintro}. We establish a comparison principle for viscosity solutions of \eqref{eq:pdeintro} proving the uniqueness of the viscosity solutions of this PDE. We then establish the most technical result of this paper concerning the uniform convergence of $\Lambda_N^f$ toward the unique viscosity solution of \eqref{eq:pdeintro} with initial condition $f\in BL(\mesTdeux)$. We finally give a complete proof of the large deviations part of Theorem \ref{thm:largedev1}. In the last section we explain how using the explicit solution of \eqref{eq:PDE} given by the value function of an optimal control problem on the probability measures, one can obtain an explicit formula of the good rate function of $(X_N)_{N\ge 1}$.

\subsection{Notation and background}
$\bullet$ In a metric space $(\mathbb X,d)$ we denote by $B(x,\gamma)$ the open ball of center $x\in\mathbb X$ and radius $\gamma\ge 0$ and $\overline{B(x,\gamma)}$ its closure.
\newline
$\bullet$ Given $(\mathbb X_k)_{k\in\N}$ a family of sets, we denote by $\pi_i:\prod_{k\in\N} \mathbb X_k\to \mathbb X_i$ the projection map on $\mathbb X_i$ given by $\pi_i((x_k)_{k\in\N})=x_i$.
\newline
$\bullet$ We shall denote by $x\cdot y$ the usual scalar product in $\R^d$ and $|z|$ the Euclidean distance of $z\in \R^d$.
\newline
$\bullet$ The $d\ge 1$ dimensional torus $\T^d=\R^d/\Z^d$ is endowed with the natural distance $|x-y|_{\T^d}=\inf_{p\in\Z^d}|x-y-p|$. We shall forget the index $\T^d$ when there is no ambiguity. We also define the exponential map $\exp :\T^d\times\R^d\to \R^d$ by $\exp(x,v)=x+v \text{ modulo } \Z^d$.
\newline
$\bullet$ Given $(\mathbb X,d)$ and $(\mathbb Y,d')$ two metric spaces which are endowed with their respective Borel $\sigma$-algebra, $\mu$ a measure on $\mathbb X$ and $T:\mathbb X\to \mathbb Y$ a measurable map we define the push forward measure  of $\mu$ by $T$ as $\mu_{\sharp}T(A)=\mu((T^{-1}(A))$ for every Borel sets $A$ of $\mathbb Y$.
\newline
$\bullet$ For $(\mathbb X,d)$ a metric space, we write $\mathcal P(\mathbb X)$ the set of probability measures on $\mathbb X$. 
\newline
Given $\mu,\nu \in\mathcal P(\mathbb X)$, we denote by $\Pi(\mu,\nu)$ the set of couplings between $\mu$ and $\nu$ that is the set of probability measures $\gamma$ on $\mathbb X\times \mathbb X$ such that $(\pi_1)_\sharp\gamma=\mu$ and $(\pi_2)_\sharp\gamma=\nu$.
\newline
For every $p\ge 1$ we define the $p$-Wasserstein space $(\mathcal P_p(\mathbb X),W_p)$ which is the set of probability measures $\mu$ on $\mathbb X$ such that there exists $x_0\in\mathbb X$, $$\int_{\mathbb X}d(x_0,x)^p \mu(dx)<+\infty.$$
endowed with the $p$-Wasserstein distance defined by $$W_p(\mu,\nu)^p=\inf_{\pi\in\Pi(\mu,\nu)}\int_{\mathbb X\times\mathbb X}d(x,y)^p \pi(dx,dy).$$
If $\mathbb X$ is a Polish space then $(\mathcal P_p(\mathbb X),W_p)$ is also a Polish space (see for instance \cite{villani2009optimal}).
\newline
$\bullet$ Given $y\in (\mathbb T^d)^N$ we define the empirical distribution of $y$ as $$\mu_y^N=\frac{1}{N}\sum_{i=1}^N\delta_{y_i}\in\mesT.$$
\newline
$\bullet$ For $f:\mathbb X\to \R^d$ a function we write $\|f\|_\infty:=\sup_{x\in \mathbb X} |f(x)|$
\newline
$\bullet$ A modulus of continuity $\omega:[0,+\infty)\to[0,+\infty)$ is a non-decreasing concave function such that $\omega(0)=0$ 
\newline
$\bullet$ If $(\mathbb X,d)$ is a Polish (i.e. complete separable  metric) space then for every $T>0$ the space $C([0,T],\mathbb X)$ is also a Polish space endowed with the supremum distance $\mathcal D(f,g)=\sup_{t\in [0,T]} d(f(t),g(t))$. 
\newline
$\bullet$ Let $(\mathbb X,d)$ be a metric space. We denote by $C_b(\mathbb X)$ the set of continuous and bounded functions from $\mathbb X$ to $\R$. We also write $BL(\mathbb X)$ the functions of $C_b(\mathbb X)$ that are Lipschitz.
\newline 
$\bullet$ A subset $\mathcal G$ of $C_b(\mathbb X)$ is said to be well-separating if
\begin{enumerate}
\item{$\mathcal G$ contains the constant functions.}
\item{$\mathcal G$ is stable by finite pointwise minima i.e. $g_1,g_2\in\mathcal G$ implies $\min(g_1,g_2)\in \mathcal G$.}
\item{$\mathcal G$ separates points of $\mathbb X$ in the sense that for every $x,y\in\mathbb X$ with $x\ne y$ and $a,b\in\R$ there exists $g\in\mathcal G$ such that $g(x)=a$ and $g(y)=b$.}
\end{enumerate}
$\bullet$ A function $f:(\T^d)^N\to \R$ is said to be symmetric if for all permutations $\sigma$ of $\{1,...,N\}$ we have $f(x_1,...,x_N)=f(x_{\sigma(1)},...,x_{\sigma(N)})$.
\newline
$\bullet$ A sequence of symmetric functions $f_N:(\T^d)^N\to \R$ converges uniformly toward $f:\mesT \to \R$ if $$\sup_{x_1,...,x_N\in(\T^d)^N}\left|f\left(\cfrac{1}{N}\sum_{i=1}^N\delta_{x_i}\right)-f_N(x_1,...,x_N)\right|\underset{N\to+\infty}{\longrightarrow} 0.$$
$\bullet$ A function $\mathcal U:\mesTdeux\to\R$ is H-differentiable at $\mu\in \mesTdeux$ if there exists $\phi\in L^2(\T^d,\R^d)$ and a modulus of continuity $\omega$ such that, for any $\gamma\in\mathcal P_2(\T^d\times\R^d)$ with $(\pi_1)_\sharp \gamma=\mu$ we have 
$$\left|\mathcal U(\exp_\sharp\gamma)-\mathcal U(\mu)-\int_{\T^d\times\R^d}\phi(x)\cdot \, v \,\gamma(dx,dv)\right|\le \omega(C_{\T^d}(\gamma))C_{\T^d}(\gamma),$$
where $C_{\T^d}(\gamma)^2=\int_{\T^d\times\R^d}|v|^2\gamma(dx,dv)$. This function $\phi$ is denoted by $D_\mu\mathcal U(\mu)$. We also recall that for $\mathcal U:\mesTdeux \to\R$ which is H-differentiable everywhere, if we define $u:(\T^d)^N\to\R$ by $u(x)=\mathcal U(\mu_x^N)$ then $u$ is differentiable everywhere and for all $x\in(\T^d)^N$ we have $$\nabla_{x_i}u(x)=\cfrac{1}{N}D_\mu\mathcal U(\mu_x^N)(x_i),$$
see for instance \cite{bertucci2026analysis,cardaliaguet2010notes}.
\newline
$\bullet$ All the random variables are defined on a same probability space $(\Omega,\mathcal F,\mP)$.

\section{Some generalities about large deviations}
In all this section, $(\mathbb X,d)$ is a Polish (complete, separable metric) space.

\subsection{Definitions from large deviations theory}
We recall some classical definitions of large deviations \cite{dembo2009large}.

\begin{definition}
A rate function is a lower semicontinuous (in short lsc) function $I: \mathbb X\to ]-\infty,+\infty]$. A rate function $I$ is called a good rate function if for all $M\in\R$, $\{x\in\mathbb X,I(x)\le M\}$ is a compact subset of $\mathbb X$.
\end{definition}
We now give the definition of a large deviations principle.
\begin{definition}
\label{def: large dev}
Let $(X_N)_{N\ge 1}$ be a family of random variables defined on a same probability space with values in $\mathbb X$ and $(\varepsilon_N)_{N\ge1}$ be a sequence of positive real numbers diverging to $+\infty$. Let also $I$ be a rate function on $\mathbb X$. 
\newline
The sequence $(X_N)_{N\ge 1}$ is said to satisfy a weak large deviations principle (in short WLDP) at speed $\varepsilon_N$ with rate function $I$ if for every open subsets $O$ of $\mathbb X$ and for every compact subsets $K$ of $\mathbb X$ we have
\begin{align*}
-\inf_{x\in O}I(x)&\le \underset{N\to +\infty}{\liminf}\cfrac{1}{\varepsilon_N}\log(\mP(X_N\in O))\\
-\inf_{x\in K}I(x)&\ge \underset{N\to +\infty}{\limsup}\cfrac{1}{\varepsilon_N}\log(\mP(X_N\in K)).
\end{align*}
\newline
The sequence $(X_N)_{N\ge 1}$ is said to satisfy a large deviations principle (in short LDP) at speed $\varepsilon_N$ with good rate function $I$  if for every open subsets $O$ of $\mathbb X$ and for every closed subsets $F$ of $\mathbb X$ we have 
\begin{align*}
-\inf_{x\in O}I(x)&\le \underset{N\to +\infty}{\liminf}\cfrac{1}{\varepsilon_N}\log(\mP(X_N \in O))\\
-\inf_{x\in F}I(x)&\ge \underset{N\to +\infty}{\limsup}\cfrac{1}{\varepsilon_N}\log(\mP(X_N\in F)).
\end{align*}
\end{definition}
\begin{remark}
\label{remark:identiI}
If the sequence $(X_N)_{N\in\N}$ satisfies a LDP with rate function $I$ we can easily check that for all $x\in \mathbb X$, $$-I(x)=\underset{\gamma\to 0}{\lim}\,\underset{N\to +\infty}{\liminf}\,\cfrac{1}{\varepsilon_N}\log(\mP(X_N \in B(x,\gamma)))=\underset{\gamma\to 0}{\lim}\,\underset{N\to +\infty}{\limsup}\,\cfrac{1}{\varepsilon_N}\log(\mP(X_N \in \overline{B(x,\gamma)})).$$
\end{remark}
\begin{remark}
Let us recall that if $(X_N)_{N\ge 1}$ satisfies an LDP at speed $\varepsilon_N$ with good rate function $I$, then choosing $\mathbb X$ as a closed set and open set in the definition of the LDP, we obtain that $\inf_\mathbb X I=0$ and we can show that this infimum is reached using that $I$ is a good rate function.
\end{remark}

\subsection{Exponential tightness and Prokhorov type theorem}
We now introduce the standard notion of exponential tightness that plays the same role in large deviations theory as the tightness for the convergence of random variables \cite{dembo2009large}.
\begin{definition}
A sequence of random variables $(X_N)_{N\ge 1}$ is said to be exponentially tight at speed $\varepsilon_N$ if for every $L>0$ there exists a compact $K_L$ of $\mathbb X$ such that 
$$\underset{N\to +\infty}{\limsup} \frac{1}{\varepsilon_N} \log (\mP(X_N\notin K_L))\le -L.$$
\end{definition}
The exponential tightness is a standard notion in the theory of large deviations since it allows to go from a WLDP to an LDP (see Section 1 of \cite{dembo2009large}). 
\begin{proposition}
Assume that $(X_N)_{N\ge 1}$ is a sequence of random variables with values in $\mathbb X$ satisfying a WLDP with rate function $I$ at speed $\varepsilon_N$ and which is exponentially tight at speed $\varepsilon_N$. Then $I$ is a good rate function and $(X_N)_{N\ge 1}$ satisfies an LDP with good rate function $I$ at speed $\varepsilon_N$. 
\end{proposition}
Another standard application of the exponential tightness is the following Prokhorov type theorem (Lemma 4.1.23 in \cite{dembo2009large}). 
\begin{theorem}
\label{thm:prokhorovlargedev}
Let $(X_N)_{N\ge 1}$ be a sequence of random variables with values in $\mathbb X$ which is exponentially tight at speed $\varepsilon_N$. Then for every subsequence we can extract a subsubsequence of $(X_N)_{N\ge 1}$ that satisfies an LDP at speed $\varepsilon_N$ with a good rate function $I$.
\end{theorem}
As a consequence, we have the following result. 
\begin{corollary}
\label{coro:prokhorovlargedev}
Let $(X_N)_{N\ge 1}$ be a sequence of random variables with values in $\mathbb X$ which is exponentially tight at speed $\varepsilon_N$. Assume that there exists a function $I$ such that for every subsequence of $(X_N)_{N\ge 1}$ one can extract a further subsequence that satisfies an LDP with good rate function $J$ and such that every such $J$ satisfies $J=I$. Then $(X_N)_{N\ge 1}$ satisfies an LDP at speed $\varepsilon_N$ with good rate function $I$.
\end{corollary}

\begin{proof}
Assume that $(X_N)_{N\ge 1}$ does not satisfy an LDP at speed $\varepsilon_N$ with good rate function $I$. For instance, assume that there exists an open subset $O$ of $\mathbb X$ such that $$-\inf_{x\in O}I(x)>\underset{N\to +\infty}{\liminf}\cfrac{1}{\varepsilon_N}\log(\mP(X_N\in O))$$ (we could do the same reasoning if the property with a closed subset failed).
By definition of the infimum limit, we can choose a small $\gamma>0$ and a subsequence $\phi_\gamma$ such that for every $N\ge 1$, $$-\inf_{x\in O}I(x)-\gamma>\cfrac{1}{\varepsilon_{\phi_{\gamma(N)}}}\log(\mP(X_{\phi_\gamma(N)}\in O)).$$
By Theorem \ref{thm:prokhorovlargedev}, we can extract a subsequence $\eta$ of this subsequence such that along this subsequence $(X_N)_{N\ge 1}$ satisfies an LDP with the good rate function $I$ thanks to the hypothesis.
Applying the large deviations property for open subsets along this subsequence yields 
$$-\inf_{x\in O}I(x)-\gamma\ge \underset{N\to +\infty}{\liminf}\cfrac{1}{\varepsilon_{\phi_{\gamma(\eta(N))}}}\log(\mP(X_{\phi_{\gamma(\eta(N)}}\in O))\ge-\inf_{x\in O}I(x).$$ 
This gives a contradiction.
\end{proof}
Another consequence of exponential tightness is the following result that allows to view large deviations as a generalization of the Laplace method in infinite dimension. This point of view can be explained through the Varadhan Lemma and the Bryc inverse Varadhan Lemma \cite{Varadhan1984LargeDA,dembo2009large}. 
\begin{theorem}[Chapter 4 \cite{dembo2009large} ]
\label{thm:largedevvaradhan}
Let $(X_N)_{N\ge 1}$ be a sequence of random variables with values in $\mathbb X$.
\newline
$\bullet$  Assume that $(X_N)_{N\ge 1}$ satisfies an LDP at speed $\varepsilon_N$ with good rate function  $I$ and let $f:\mathbb X\to \R$ be a continuous function. Assume furthermore that 
$$\underset{M\to +\infty}{\lim}\underset{N\to +\infty}{\lim}\cfrac{1}{\varepsilon_N}\log\left(\E\left(\exp\left(\varepsilon_N f(X_N)\right)\mathds{1}_{f(X_N)\ge M}\right)\right)=-\infty.$$
Then we have $$\underset{N\to +\infty}{\lim}\cfrac{1}{\varepsilon_N}\log\left(\E\left(\exp\left(\varepsilon_N f(X_N)\right)\right)\right)=\sup_{x \in \mathbb X}\{f(x)-I(x)\}.$$
$\bullet$ Conversely, assume that $(X_N)_{N\ge 1}$ is exponentially tight at speed $\varepsilon_N$ and that for every $f\in C_b(\mathbb X)$ the following limit exists $$\Lambda^f:=\underset{N\to +\infty}{\lim}\cfrac{1}{\varepsilon_N}\log\left(\E\left(\exp\left(\varepsilon_N f(X_N)\right)\right)\right).$$ 
Then, $(X_N)_{N\ge 1}$ satisfies an LDP at speed $\varepsilon_N$ with good rate function 
$$I(x)=\sup_{f\in C_b(\mathbb X)}\{f(x)-\Lambda^f\}.$$
Moreover, if $\Lambda^f$ exists for every $f\in \mathcal G$ with $\mathcal G$ a well-separating class of functions on $\mathbb X$, then $\Lambda^f$ exists for every $f\in C_b(\mathbb X)$ and so the previous conclusions hold.  
\end{theorem}
\begin{remark}
A standard example of such well-separating $\mathcal G$ that shall be mentioned later on is the set $BL(\mathbb X)$ of Lipschitz and bounded functions on $\mathbb X$.
\end{remark}
\subsection{Large deviations for continuous in time processes}
\label{sec:Large deviations for continuous in time processes}
We can obtain an LDP for random variables with values in $\mathbb Y=C([0,T],\mathbb X)$ if we know that the finite dimensional marginals satisfy an LDP. This result is a consequence of what is known is the literature as the abstract LDP for projective limits (see Section 4.6 of \cite{dembo2009large}). 
We detail more concretely for this particular case how it works for the sake of completeness of the reader following the ideas of the Chapter 5 of \cite{feng2006large} in Appendix \Ref{Appendix:A}. 
\begin{theorem}
\label{thm:ldpmarginalstoproces}
Let $(X_N)_{N\ge 1}$ be a sequence of random variables with values in $\mathbb Y$ which is exponentially tight at speed $\varepsilon_N$ and assume that for all $m\ge 1$, and $0\le t_1<...<t_m\le T$ the marginals $(X_N(t_1),...,X_N(t_m))_{N\ge 1}$ satisfy an LDP in $\mathbb X^m$ at speed $\varepsilon_N$ with good rate function $I_{t_1,...,t_m}$. Then $(X_N)_{N\ge 1}$ satisfies an LDP in $\mathbb Y$ at speed $\varepsilon_N$ with good rate function 
\begin{equation}
\label{eq:ratefunctionprocess}
I(f)=\sup_{1\le i\le m}\sup_{0\le t_1<...<t_m\le T} I_{t_1,...,t_m}(f(t_1),...,f(t_m)).
\end{equation}
\end{theorem}

\subsection{Large deviations on product spaces.}
\label{sec:Large deviations on product spaces.}
As explained in Section \ref{sec:Large deviations for continuous in time processes}, to study the rate function of a family of continuous in time random variables we only have to study the rate function of its finite dimensional marginals. 
By the Varadhan/Le Bryc Principle stated in Theorem \ref{thm:largedevvaradhan}, to characterize the good rate function for the LDP of finite dimensional marginals of a sequence of random variables $(X_N)_{N\ge 1}$ with values in $C([0,T],\mathbb X)$, it is sufficient to show that for every $m\ge 1$, for every $F\in C_b(\mathbb X^m)$ that are Lipschitz continuous the following quantity is uniquely characterized
\begin{equation}
\label{eq:laplacetransfofinidimcomplet}
\underset{N\to +\infty}{\lim}\cfrac{1}{\varepsilon_N}\log\left(\E\left(\exp\left(\varepsilon_N F(X(t_1),...,X(t_m))\right)\right)\right).
\end{equation}
However, we shall justify in this section that we can actually restrict to a certain class of test functions $F$ to characterize the rate function of the finite dimensional marginals. Indeed, for instance in some good spaces such as $\mathbb X=\R^d$ it is known that the linear functions $F_\lambda(x)=\lambda\cdot x$ for $\lambda\in (\R^d)^m$ are sufficient.
\newline
Here, we shall prove that we can restrict the study to the test functions $F$ of the form $F(x_1,...,x_m)=\sum_{i=1}^{m} f_i(x_i)$ where $f_i\in BL(\mathbb X)$ for all $1\le i\le m$. 
\begin{proposition}
\label{prop:largedevmarginals}
Let $(X_N^1,...,X_N^m)_{N\ge 1}$ be a family of random variables with values in $\mathbb X^m$. Assume that $(X_N^1,...,X_N^m)_{N\ge 1}$ satisfies an LDP with good rate function $I_m:\mathbb X^m\to[0,+\infty]$ at speed $\varepsilon_N$.
Then, for every $f_1,...,f_m\in BL(\mathbb X)$ we have
\begin{equation}
\label{eq:varhandanfinidim}
\begin{split}
\Lambda_m(f_1,...,f_m):=&\underset{N\to +\infty}{\lim}\cfrac{1}{\varepsilon_N}\log\left(\E\left(\exp\left(\varepsilon_N f_1(X_N^1)+...+\varepsilon_N f_m(X_N^m)\right)\right)\right)\\=&\sup_{x_1,...,x_m\in\mathbb X}\left[\sum_{i=1}^m f_i(x_i)-I_m(x_1,...,x_m)\right].
\end{split}
\end{equation}
Moreover, the knowledge of $\Lambda_m(f_1,...,f_m)$ for every $f_1,...,f_m\in BL(\mathbb X)$ uniquely characterizes $I_m$ and we can more explicitly write the inversion formula for every $(x_i)_{1\le i\le m}\in\mathbb X^m$
\begin{equation}
\label{eq:invformula}
I_m(x_1,...,x_m)=\sup_{f_1,...,f_m \in BL(\mathbb X)}\left[\sum_{i=1}^m f_i(x_i)-\Lambda_m(f_1,...,f_m)\right]
\end{equation}
\end{proposition}
Let us mention that this result is not an immediate consequence of the Varadhan/Le Bryc principle since this set of test functions is not stable by pointwise minima and so is not well-separating. Nonetheless this result can also be viewed as a consequence of Section 3.3 of \cite{feng2006large}. We give an elementary  proof of this result in Appendix \ref{Appendix:B}.

\section{A viscosity solution approach to large deviations of particle systems}
We recall that $b:\T^d\times\mesTdeux\to\R^d$ is Lipschitz with constant $K_b>0$ and $\sigma_N>0$ goes to 0 when $N$ goes to infinity. We are interested in the following system of particles on $\T^d$
\begin{equation}
\label{eq:systemofparticles}
dx_t^i=b(x_t^i,\mu_t^N)dt+\sigma_N dB_t^i \mathrm{ \,\,modulo\,\, }\Z^d, \, 1\le i\le N
\end{equation}
where $(B^i)_{i\in\N}$ are iid standard Brownian motions and $\mu_t^N=N^{-1}\sum_{1\le i\le N}\delta_{x_t^i}$ is the empirical distribution of the particles.
\newline
As explained in the introduction, we shall prove an LDP at speed $\varepsilon_N:=N/\sigma_N^2$ of the sequence of random variables $(X_N)_{N\ge 1}$ defined by 
\begin{equation}
\label{eq:randomvariables}
\begin{split}
X_N \,:\, (\Omega,\mathcal F ,\mP)&\to C([0,T],\mesTdeux)\\
\omega&\mapsto \left(t\mapsto \frac{1}{N}\sum_{i=1}^N\delta_{x_t^i(\omega)}\right),
\end{split}
\end{equation}
where $(x_t^i)_{1\le i\le N}$ is a strong solution of \eqref{eq:systemofparticles}.
We shall follow a usual strategy as it was also presented in Chapter 3, 4 and 5 of \cite{feng2006large} or in \cite{kraaij2018large}.
We first show in Section \ref{subsection:exptightness} that $(X_N)_{N\ge 1}$ is exponentially tight at speed $N/\sigma_N^2$ in $C([0,T],\mesTdeux)$.
Having in mind Corollary \ref{coro:prokhorovlargedev}, since $(X_N)_{N\ge 1}$ is exponentially tight, it suffices to show the uniqueness of the good rate function up to extraction to prove the complete LDP. To lighten notation, we will forget the extraction and we shall assume that $(X_N)_{N\ge 1}$ satisfies an LDP with good rate function $I$ and we shall characterize uniquely $I$.
\newline
First, since the LDP of marginals characterized the LDP of the process by Theorem \ref{thm:ldpmarginalstoproces}, to prove the uniqueness of $I$ it suffices to show the uniqueness of good rate functions for the LDP of finite dimensional marginals. 
\newline
Then, by the Varadhan/Le Bryc Principle stated in Theorem \ref{thm:largedevvaradhan} and Proposition \ref{prop:largedevmarginals}, to prove the uniqueness of the good rate function for the LDP of finite dimensional marginals, it is sufficient to show that for every $m\ge 1$, for every $f_1,...,f_m\in C_b(\mesTdeux)$ that are Lipschitz the following quantity exists and is uniquely determined 

\begin{equation}
\label{eq:laplacetransfofinidim}
\underset{N\to +\infty}{\lim}\cfrac{1}{\varepsilon_N}\log\left(\E_x\left[\exp\left(\varepsilon_N f_1(X_N(t_1))+...+\varepsilon_N f_m(X_N(t_m))\right)\right]\right).
\end{equation}

\subsection{Exponential tightness of the system of particles}
\label{subsection:exptightness}
The goal of this section is to prove that $(X_N)_{N\ge 1}$ is exponentially tight at speed $N/\sigma_N^2$. The proof is based on a concentration result for system of particles in mean field interaction as it was also studied in \cite{delarue2020master}.
\newline
First let us give some notation. For $x\in\T^d$ we define, in this section only, the modulus of its sinus by $$|\sin(x)|=\pi^{-1}\left(\sum_{k=1}^d|\sin(\pi x_k)|^2\right)^{1/2}.$$ 
Let $\zeta:\T^d\times\T^d\to \R^d$ be the unique element of $[-2^{-1},2^{-1})^d$ such that $\zeta(x,y)=x-y-k(x,y)$ where $k(x,y)\in\Z^d$. In particular we have $|\zeta(x,y)|=d_{\T^d}(x,y)$ for all $(x,y)\in(\T^d)^2$.
\newline
For all $(x,y)\in(\T^d)^2$ we have $|\sin(y-x)|=|\sin(\zeta(x,y))|$ and using that $\zeta(x,y)\in [-2^{-1},2^{-1})^d$ and the inequality $$\frac{2}{\pi}|x|\le |\sin(x)|\le |x|$$ for $x\in [0,\pi/2]$ we obtain that there exists a constant $C>0$ which only depends on $d$ such that for all $(x,y)\in(\T^d)^2$
\begin{equation}
\label{eq:equivalencenorme} 
C^{-1} d_{\T^d}^2(x,y)\le |\sin(x-y)|^2\le C d_ {\T^d}^2(x,y).
\end{equation}
Hence $\tilde d_{\T^d}(x,y)=|\sin(x-y)|$ defines a distance equivalent to the standard distance on $\T^d$ but which is more regular ($ (\tilde d_{\T^d})^2$ is a smooth function).
\newline
Moreover, in this section we write $(\tilde x_t^i)_{t\ge 0,\,1\le i\le N}$ the solution of the stochastic differential equations $$d\tilde x_t^i=b(x_t^i,\mu_{x_t}^N)dt+\sigma_N dB_t^i,\, 1\le i\le N$$
where $(B^i)$ is a family of iid standard Brownian motions in $\R^d$ and $x_t^i\in\T^d$ which satisfies $x_t^i=\tilde x_t^i \text{ modulo }\Z^d$ for all $t\ge 0$ and $1\le i\le N$.
\newline
In order to prove the exponential tightness of $(X_N)_{N\ge 1}$ we first explicit compact subsets of $C([0,T],\mesTdeux)$. The following lemma is a direct consequence of the Arzela-Ascoli theorem. 
\begin{lemma}
\label{lemma:compactsubset}
For every $L>0$, the following sets are compact subsets of $C([0,T],\mesTdeux)$
\begin{align*}K_L&:=\bigcap_n K_{L,n}\\
&:=\bigcap_n\left\{\mu\in C([0,T],\mesTdeux)\,\Big|\,\sup_{|t-s|\le\delta_n}W_2^2(\mu(t),\mu(s))\le L\gamma_n\right\},
\end{align*} with $\delta_n$ and $\gamma_n$ two positive sequences that goes to 0. 
\end{lemma}
We now state a concentration lemma to deal with oscillations of the process $(X_N)_{N\ge 1}$.
\begin{lemma}
\label{prop:exptightness}
There exists $M_0$ that only depend on $b$, $d$  and $\sup_{N\ge 1}\sigma_N$ and constants $C_1,C_2>0$ that only depends on $d$ and $T$ such that for every $\delta\le \min(1,T)$, $0<\alpha<1$, $N\ge 1$ and for every $M\ge M_0$ we have 
$$\mP\left(\sup_{|t-s|\le \delta}W_2(X_N(t),X_N(s))^2\ge 18M\delta^\alpha\right)\le \cfrac{C_1}{\delta}\exp\left(-\frac{C_2 N M^2\delta^{2\alpha-1}}{\sigma_N^2}\right).$$ 
\end{lemma}
\begin{proof}
In this proof $C>0$ is a constant that only depends on $d$ and that will be updated along the proof. We shall cut the interval $[0,T]$ in intervals of size $\delta$; $0=t_0<t_1=\delta<...< t_k=k\delta\le t_{k+1}=T$ with $k=\lfloor \frac{T}{\delta}\rfloor$.
\newline
For all $0\le u\le T$, $0\le k\le K$ and $1\le i\le N$, we define for $u\in[t_k,t_{k+1}]$ $$z_u^i=(2\pi)^{-1}(\sin(2\pi(\tilde x_u^{j,i}-\tilde x_{t_k}^{j,i}))_{1\le j\le d}.$$
That there exists a constant $C$ that only depends on $d$ such that for all $0\le k\le K$, $\sup_{t_k\le u\le T,\,1\le i\le N}|z_u^i|\le C$.
\newline
First we notice that 
$$\sup_{|t-s|\le\delta} W_2(X_N(t),X_N(s))^2\le 9 \max_{0\le k\le K}\sup_{t\in[t_k,t_{k+1}]}W_2(X_N(t),X_N(t_k))^2.$$
Applying Itô's formula and using \eqref{eq:equivalencenorme} yields for $t\in[t_k,t_{k+1}]$
\begin{align*}
W_2(X_N(t),X_N(t_k))^2&\le \cfrac{C}{N}\sum_{i=1}^N|\sin(\tilde{x}_t^i-\tilde{x}_{t_k}^i)|^2\\
&\le\cfrac{C}{N}\sum_{i=1}^N\left[ 2\int_{t_k}^t z_u^i.b(x_u^i,\mu_u^N)du+2\sigma_N\int_{t_k}^t z_u^i.dB_u^i+\sigma_N^2(t-t_k)\right]\\
&\le M_0 |t_k-t|+\cfrac{C\sigma_N}{N}\sum_{i=1}^N\int_{t_k}^t z_u^i.dB_u^i,
\end{align*}
where $M_0$ is a constant that only depends on $d$, $b$ and $\sup_N \sigma_N$.
\newline 
Hence, we get 
\begin{align*}
\mP\left(\sup_{|t-s|\le \delta}W_2(X_N(t),X_N(s))^2\ge 18 M\delta^\alpha\right)&\le\\
& \sum_{k=0}^K\mP\left(\sup_{t_k\le t\le t_{k+1}}\left|\cfrac{C\sigma_N}{N}\sum_{i=1}^N\int_{t_k}^t z_u^i.dB_u^i\right|\ge M\delta^\alpha\right) \\
&\le \cfrac{C_1}{\delta}\exp\left(-\frac{C_2 N M^2\delta^{2\alpha-1}}{\sigma_N^2}\right),
\end{align*}
where the last inequality comes from standard concentration inequality (see for instance Corollary H.13 and Lemma 4.3.14 of \cite{anderson2010}).
\end{proof}

\begin{theorem}
\label{thm:exptightness}
The sequence $(X_N)_{N\ge 1}$ is exponentially tight in $C([0,T],\mesTdeux)$ at speed $N/\sigma_N^2$.
\end{theorem}
\begin{proof}
Given $L>0$, our goal is to construct a compact $K_L$ of $C([0,T],\mesTdeux)$ such that $$\underset{N\to +\infty}{\limsup} \frac{\sigma_N^2}{N} \log (\mP(X_N\notin K_L))\le -L.$$
We first choose $n_0$ (which only depends on $T$) such that $n_0^{-2}\le \min(1,T)$.
We choose a compact of the form $K_l=\cap_{n\ge n_0}K_{l,n}$ as in Lemma \ref{lemma:compactsubset} with $l$, $\delta_n$ and $\gamma_n$ that shall be chosen later on. 
We get 
\begin{align*}
\mP(X_N\notin K_l)&\le \sum_{n\ge n_0}\mP(X_N\notin K_{l,n})\\
&\le \sum_{n\ge n_0}\mP\left(\sup_{|t-s|\le\delta_n} W_2(X_N(t),X_N(s))^2\ge l\gamma_n\right)
\end{align*}
We choose $l=18M$ (where $M$ shall be chosen later on), $\delta_n=1/n^2$ and $\gamma_n=\delta_n^{1/4}$ and applying Lemma \ref{prop:exptightness} for $M\ge M_0$ and for every $n\ge 1$ yields 
$$\mP\left(\sup_{|t-s|\le n^{-2}}W_2(X_N(t),X_N(s))^2\ge \frac{18M}{\sqrt{n}}\right)\le C_1 n^2\exp\left(\frac{-C_2 \,N M^2n}{\sigma_N^2}\right), $$
where the constants $C_1,C_2>0$ is given in Lemma \ref{prop:exptightness}. 
Hence, we get for $M\ge M_0$
\begin{align*}
\mP(X_N\notin K_{18M})&\le C_1\sum_{n\ge n_0}n^2\exp\left(\frac{-C_2 \,N M^2\,n}{\sigma_N^2}\right).
\end{align*}
A direct computation shows that $$\lim_N \cfrac{\sigma_N^2}{N}\log\left(C_1\sum_{n\ge n_0}n^2\exp\left(\frac{-C_2 \,N M^2\,n}{\sigma_N^2}\right)\right)=-C_2n_0 M^2.$$
So we deduce that for every $M\ge M_0$  $$\underset{N\to +\infty}{\limsup} \frac{\sigma_N^2}{N} \log (\mP(X_N\notin K_{18M}))\le -C_2 n_0 M^2.$$
Choosing $M$ such that $C_2 n_0 M^2=L$ concludes the proof.
\end{proof}

\subsection{Semigroup property}
\label{section:semigroupprop}
\begin{definition}
\label{def:semigroup}
For every $N\ge 1$, we define the transfer function of $X_N$ by $$\Lambda_{N}^f(t,y)=\cfrac{1}{\varepsilon_N}\log\left(\E\left[\exp\left(\varepsilon_N f(X_N(t)\right)|x_0^i=y_i, \forall 1\le i\le N\right]\right), $$
for every $y\in(\mathbb T^d)^N$ and $f\in C_b(\mesTdeux)$.
\end{definition}
\begin{remark}
Let us notice that $f$ can be viewed as the initial condition of $\Lambda_N^f$ since $\Lambda_N^f(0,y)=f(\mu_y^N)$.
\end{remark}
We can now express the quantity \eqref{eq:laplacetransfofinidim} only in term of $\Lambda_N$.
\begin{proposition}
\label{prop:markovproperty}
Let $f_1,...,f_m\in C_b(\T^d)$ be $m\ge 1$ Lipschitz functions and $0\le t_1<t_2<...<t_m\le T$. Then for every $N\ge 1$ and $y\in(\mathbb T^d)^N$, we have the following property
\begin{equation}
\label{eq:towerprop}
\begin{split}
\cfrac{1}{\varepsilon_N}&\log\left(\E\left[\exp\left(\varepsilon_N f_1(X_N(t_1))+...+\varepsilon_N f_m(X_N(t_m))\right)|x_0^i=y_i, \forall 1\le i\le N\right]\right)\\
&=\Lambda_N^{f_1+\Lambda_N^{f_2+...}(t_2-t_1,\bullet)}(t_1,y).
\end{split}
\end{equation}
In an implicit form we can rewrite the previous formula as 
$$\cfrac{1}{\varepsilon_N}\log\left(\E\left[\exp\left(\varepsilon_N f_1(X_N(t_1))+...+\varepsilon_N f_m(X_N(t_m))\right)|x_0^i=y_i, \forall 1\le i\le N\right]\right)=\psi_1^N(y),$$
where the family $(\psi_k^N)_{1\le k\le m}$ is defined recursively backward by
$$\psi_m^N(y) = \Lambda_N^{f_m}(t_m-t_{m-1},y),$$
and, for $k\le m-1$,
$$\psi_k^N(y)=\Lambda_N^{\,f_k+\psi_{k+1}^N}\bigl(t_k-t_{k-1},y\bigr),\qquad t_0=0.$$
\end{proposition}
\begin{proof}
This is a straightforward consequence of the Markov Property. 
Let us explicit the case $m=2$ since the general case can be done similarly. Conditioning by the value of $X_N(t_1)$ we get
\begin{align*}
\cfrac{1}{\varepsilon_N}&\log\left(\E\left[\exp\left(\varepsilon_N f_1(X_N(t_1))+\varepsilon_N f_2(X_N(t_2))\right)|x_0^i=y_i, \forall 1\le i\le N\right]\right)\\
&=\cfrac{1}{\varepsilon_N}\log\left(\E\left[\E\left[\exp\left(\varepsilon_N f_1(X_N(t_1))+\varepsilon_N f_2(X_N(t_2))\right)|X(t_1)\right]|x_0^i=y_i, \forall 1\le i\le N\right]\right)\\
&=\cfrac{1}{\varepsilon_N}\log\left(\E\left[\exp\left(\varepsilon_N f_1(X_N(t_1))\E\left[\varepsilon_N f_2(X_N(t_2))\right)|X(t_1)\right]|x_0^i=y_i, \forall 1\le i\le N\right]\right)\\
&=\cfrac{1}{\varepsilon_N}\log\left(\E\left[\exp\left(\varepsilon_N f_1(X_N(t_1))\right)\exp\left(\varepsilon_N\Lambda_N^{f_2}(t_2-t_1,X_N(t_1))\right)|x_0^i=y_i, \forall 1\le i\le N\right]\right)\\
&=\Lambda_N^{f_1+\Lambda_N^{f_2}(t_2-t_1,\bullet)}(t_1,y).
\end{align*}
\end{proof}
As explained at the beginning of this section, our goal is to show that for every $m\ge 1$, for every $f_1,...,f_m\in BL(\mesTdeux)$ the following quantity exists and to characterize it.
\begin{equation}
\label{eq:laplacetransfofinidim2}
\underset{N\to +\infty}{\lim}\cfrac{1}{\varepsilon_N}\log\left(\E\left(\exp\left(\varepsilon_N f_1(X_N(t_1))+...+\varepsilon_N f_m(X_N(t_m))\right)\right)\right).
\end{equation}
Thanks to Proposition \ref{prop:markovproperty}, by bootstrapping we have the following observation.
\begin{claim}
It suffices to characterize the convergence of $\Lambda_N^{g_N}$ where $g_N:(\mathbb T^d)^N\to \R$ is symmetric and converges uniformly toward a Lipschitz function $g:\mesTdeux\to \R$ to characterize \eqref{eq:laplacetransfofinidim2}.
\end{claim}
In next sections we shall study the convergence of the rate function in order to completely prove the LDP of $(X_N)_{N\ge 1}$.
\subsection{Formal convergence of the transfer function}
Starting from now, we fix $f:\mesTdeux\to \R$ a Lipschitz function with Lipschitz constant $L_f$ and we want to study the convergence of $(\Lambda_N^f)_{N\ge 1}$.
\newline
The generator of the system of particles is given by $$\mathcal L_N \phi(y)=\sum_{i=1}^Nb(y_i,\mu_y^N).\nabla_{y_i}\phi(y)+\cfrac{\sigma_N^2}{2}\sum_{i=1}^N\Delta_{y_i}\phi(y).$$
By the Feynman-Kac formula, $g_N^f:=\E\left[\exp\left(\varepsilon_N f(X_N(t)\right)|x_0^i=y_i, \forall 1\le i\le N\right]$ is  solution of the following linear PDE:
\begin{equation}
\begin{split}
\left\{ 
  \begin{array}{l} 
    \partial_t g_N^f(t,y)=\mathcal L_N g_N^f,\, \, t>0,\, y\in(\T^d)^N  \\ 
    g_N(0,y)=\exp\left(\varepsilon_N f\left(\mu_y^N\right)\right)
  \end{array} 
\right.
\end{split}
\end{equation}
Using the Cole-Hopf transform $\Lambda_N^f=\frac{1}{\varepsilon_N}\log(g_N^f)$ we obtain that the transfer function satisfies the following Hamilton-Jacobi equation
\begin{equation}
\label{eq:edphofcolepasfinal}
\begin{split}
\left\{ 
  \begin{array}{l} 
    \partial_t \Lambda_N^f(t,y)=\dis\sum_{i=1}^Nb(y_i,\mu_y^N).\nabla_{y_i}\Lambda_N^f(t,y)+\cfrac{\sigma_N^2}{2}\dis\sum_{i=1}^N\Delta_{y_i}\Lambda_N^f(t,y)\\
   \hspace{3.7cm} +\cfrac{\varepsilon_N\sigma_N^2}{2}\dis\sum_{i=1}^N|\nabla_{y_i}\Lambda_N^f(t,y)|^2,\, \, t>0,\, y\in(\T^d)^N  \\ 
    \Lambda_N^f(0,y)=f(\mu_y^N),\,\,y\in (\T^d)^N
  \end{array} 
\right.
\end{split}
\end{equation}
Since $\varepsilon_N=N/\sigma_N^2$, the rate function is solution of 
\begin{equation}
\label{eq:edphofcole}
\begin{split}
\left\{ 
  \begin{array}{l} 
    \partial_t \Lambda_N^f(t,y)=\dis\sum_{i=1}^Nb(y_i,\mu_y^N).\nabla_{y_i}\Lambda_N^f(t,y)+\cfrac{\sigma_N^2}{2}\dis\sum_{i=1}^N\Delta_{y_i}\Lambda_N^f(t,y)\\
   \hspace{4.cm} +\cfrac{N}{2}\dis\sum_{i=1}^N|\nabla_{y_i}\Lambda_N^f(t,y)|^2,\, \, t>0,\, y\in(\T^d)^N  \\ 
    \Lambda_N^f(0,y)=f(\mu_y^N),\,\,y\in (\T^d)^N
  \end{array} 
\right.
\end{split}
\end{equation}
More precisely since $\sigma_N>0$ and $b$ is Lipschitz, $\Lambda_N^f$ is by classical arguments a strong solution of \eqref{eq:edphofcole} in $$C\left([0,\infty)\times (\mathbb{T}^d)^N\right)\bigcap C^{1,2}\left((0,\infty)\times (\mathbb{T}^d)^N\right).$$
Formally, assume that $N$ is large and $\mu_y^N\approx \mu\in \mesT$ and $\Lambda_N^f(t,y)\approx \mathcal U(t,\mu)$. Then we formally have $\partial_t\Lambda_N^f(t,y)\approx \partial_t \mathcal U(t,\mu)$ and $\nabla_{y_i}\Lambda_N^f(t,y)\approx\frac{1}{N}\, D_\mu \mathcal U(t,\mu,y_i)$. Using these approximations in \eqref{eq:edphofcole} and that $(\sigma_N)_{N\ge1}$ goes to 0, we formally obtain that $\mathcal U$ is solution of
\begin{equation}
\begin{split}
\left\{ 
  \begin{array}{l} 
    \partial_t \mathcal U(t,\mu)\approx\cfrac{1}{N}\dis\sum_{i=1}^Nb(y_i,\mu_y^N).D_\mu\mathcal U(t,\mu,y_i)+\cfrac{1}{2N}\dis\sum_{i=1}^N|D_\mu \mathcal U(t,\mu,y_i)|^2,\, \, t>0,\, \mu\in\mesTdeux \\ 
    \mathcal U(0,\mu)=f(\mu),\,\,\mu\in \mesTdeux
  \end{array} 
\right.
\end{split}
\end{equation}
So taking $N$ large, formally the Hamilton-Jacobi equation on the space of probability measures that we obtain at the limit is 
\begin{equation}
\label{eq:limipdeformal}
\begin{split}
\left\{ 
  \begin{array}{l} 
    \partial_t \mathcal U(t,\mu)=\dis\int_{\T^d}b(y,\mu).D_\mu\mathcal U(t,\mu,y)\mu(dy)+\cfrac{1}{2}\dis\int_{\T^d}|D_{\mu}\mathcal U(t,\mu,y)|^2\mu(dy)\\
    \hspace{9.5cm} \, t>0,\, \mu\in\mesTdeux \\ 
    \mathcal U(0,\mu)=f(\mu),\,\,\mu\in \mesTdeux
  \end{array} 
\right.
\end{split}
\end{equation}

\subsection{Viscosity solutions of the limit PDE}
In this section we follow the notation and ideas of \cite{bertucci2026analysis} to study the viscosity solutions of the following PDE set on the space of probability measures \begin{equation}
\label{eq:PDE}
\begin{split}
\partial_t \mathcal U(t,\mu)=\dis\int_{\T^d}b(y,\mu).D_\mu\mathcal U(t,\mu,y)\mu(dy)+\cfrac{1}{2}\dis\int_{\T^d}|D_{\mu}\mathcal U(t,\mu,y)|^2\mu(dy),\, \, t>0,\, \mu\in\mesTdeux .
\end{split}
\end{equation} 
\subsubsection{Definition of viscosity solutions}
The torus $\T^d$ is a smooth manifold but there can be many geodesics that connects $x$ and $y$ in $\T^d$. We recall that the exponential map defined on $\T^d\times \R^d$ is the value at time $t=1$ of the unique geodesic $\theta:[0,1]\to \T^d$ such that $\theta(0)=x$ and $\theta'(0)=v$.
On the torus we explicitly have $\exp(x,v)=x+v \text{ modulo } \Z^d$.
The cost of transport along a plan $\eta\in\mathcal P_2(\T^d\times\R^d)$ is given by $$C_{\T^d}(\eta)^2=\int_{\T^d\times\R^d}|v|^2\eta(dx,dv).$$
\begin{definition}
\label{def:differentialtore}
Let $\mathcal U:\mesT\to \R$ be an upper semicontinuous function (in short usc) and $\mathcal V:\mesT\to \R$ be a lower semicontinuous function (in short lsc). A measurable function $\psi:\T^d\to\mes$ is an element of the super-differential at $\mathcal U$ at $\mu\in \mesT$ if 
\newline
$\bullet$ $$\int_{\T^d}\int_{\R^d}|z|^2\psi_x(dz)\mu(dx)<+\infty.$$
$\bullet$ There exists a modulus of continuity $\omega$ such that for any $\eta\in\mathcal P_2(\T^d\times\R^d)$ such that $(\pi_1)_\sharp \eta=\mu$, for any $\Gamma\in \mathcal P(\T^d\times\R^d\times \R^d)$ such that $(\pi_1,\pi_2)_\sharp\Gamma=\mu(dx)\psi_x(dz)$ and $(\pi_1,\pi_3)_\sharp\Gamma=\eta$, 
\begin{equation}
\label{eq:subdiffdef}
\mathcal U(\exp_\sharp\eta)\le \mathcal U(\mu)+\int_{\T^d\times\R^d\times\R^d}z.v\,\Gamma(dx,dz,dv)+C_{\T^d}(\eta)\omega(C_{\T^d}(\eta)).
\end{equation}
The set of such $\psi$ is denoted by $\partial^+\mathcal U(\mu)$. If $\partial^+\mathcal U(\mu)$ is not empty, we say that $\mathcal U$ is super-differentiable at $\mu$.
\newline
The sub-differential of $\mathcal V$ is similarly defined as $\partial^-\mathcal V(\mu)=\{x\mapsto(-Id)_\sharp\psi_x|\, \psi\in \partial^+\mathcal (-V)(\mu)\}.$ If $\partial^-\mathcal V(\mu)$ is not empty, we say that $\mathcal V$ is sub-differentiable at $\mu$.
\end{definition}
\begin{remark}
\label{rem:extensiontemps}
Let $\mathcal U:[0,T]\times \mesT\to \R$ and $\mathcal V:[0,T]\times \mesT\to \R$. We similarly say that $(\theta,\psi)$ with $\theta\in\R$ and $\psi:\T^d\to\mes$ is an element of the super-differential $\partial^+\mathcal U(t,\mu)$ of $\mathcal U$ at $(t,\mu)\in [0,T]\times \mesT$ if there exists a modulus of continuity $\omega$ such that for any $t'\in [0,T]$ and $\eta\in\mathcal P_2(\T^d\times\R^d)$ such that $(\pi_1)_\sharp \eta=\mu$, for any $\Gamma\in \mathcal P(\T^d\times\R^d\times \R^d)$ such that $(\pi_1,\pi_2)_\sharp\Gamma=\mu(dx)\psi_x(dz)$ and $(\pi_1,\pi_3)_\sharp\Gamma=\eta$, 

\begin{align*}
\mathcal U(t',\exp&_\sharp\eta)\le \mathcal U(t,\mu)+\theta(t'-t)+\\
&\int_{\T^d\times\R^d\times\R^d}z.v\,\Gamma(dx,dz,dv)+\left(\sqrt{|t'-t|^2+C_{\T^d}(\eta)^2}\right)\omega\left(\sqrt{|t'-t|^2+C_{\T^d}(\eta)^2}\right).
\end{align*}
We also define $\partial^-\mathcal V(t,\mu)=\{(-\theta,x\mapsto(-Id)_\sharp\psi_x)|\, (\theta,\psi)\in \partial^+\mathcal (-V)(\mu)\}.$
\end{remark}
Recall that $\zeta:\T^d\times\T^d\to\R^d$ as $\zeta(x,y)=x-y-k(x,y)$ where $k(x,y)\in\Z^d$ has been chosen such that $\zeta(x,y)\in [-2^{-1},2^{-1})^d$ and so in particular $|\zeta(x,y)|=|x-y|_{\T^{d}}$. Let us notice that $\zeta(x,y)$ is a constant speed geodesic joining $y$ to $x$ in time 1 and is measurable.
\newline
We now explicit some elements of the super-differential of $W_2^2$ in view of the doubling of variables technique to establish comparison principles for viscosity solutions. 
\begin{proposition}[Proposition 5.43 in \cite{bertucci2026analysis}]
\label{prop:subdiffW}
For any $\nu_0,\mu_0\in\mesTdeux$, the map $\mu\mapsto W_2(\mu,\nu_0)^2$ is super-differentiable at $\mu_0$ and if $\gamma_0$ is an optimal coupling in $W_2(\mu_0,\nu_0)$, then $$(\pi_1,2\zeta)_\sharp \gamma_0\in \partial^+W_2(.,\nu_0)^2(\mu_0).$$
\end{proposition}
We give a proof for the sake of completeness. 
\begin{proof}
Let $\eta\in \mathcal P_2(\T^d\times\R^d)$ such that $(\pi_1)_\sharp\eta=\mu_0$ and $\Gamma\in\mathcal P(\T^d\times \R^d\times\R^d)$ such that $(\pi_1,\pi_2)_\sharp\Gamma=(\pi_1,2\zeta)_\sharp \gamma_0$ and $(\pi_1,\pi_3)_\sharp\Gamma=\eta$.
Let us notice that for all $x\in\T^d$, $v\in \R^d$ and $y\in\T^d$ we have
\begin{equation}
\label{eq:majorationdistancetore}
\begin{split}
|(\exp(x,v)-y|_{\T^d}^2&=\inf_{k\in\Z^d}|x+v-y-k|^2\\
&\le |x+v-y-k(x,y)|^2\\
&\le |v|^2+|x-y-k(x,y)|^2+ 2v\cdot (x-y-k(x,y)) \\
&=|v|^2+|x-y|_{\T^d}^2+ 2v\cdot \zeta(x,y).
\end{split}
\end{equation}
The idea is to integrate this inequality against a good measure constructed from $\Gamma$. Since $\Gamma$ depends on the variables $(x,z,v)$ and this inequality depends on the variables $(x,y,v)$ we shall construct a measure that extends $\Gamma$ and which depends on the variables $(x,y,z,v)$. 
\newline
First, we define the map $T:\T^d\times\T^d\to \T^d\times\T^d\times\R^d$ by $T(x,y)=(x,y,2\zeta(x,y))$ and we construct the measure $\Lambda=T_\sharp \gamma_0$ which a measure supported on $\{(x,y,z)\,|\, z=2\zeta(x,y)\}$ by construction. 
\newline 
We notice that $(\pi_1,\pi_3)_\sharp \Lambda=(\pi_1,2\zeta)_\sharp \gamma_0=(\pi_1,\pi_2)_\sharp\Gamma$. 
By the Gluing lemma, we can construct a probability measure $\Theta$ on $\T^d\times\T^d\times\R^d\times\R^d$ such that $(\pi_1,\pi_3,\pi_4)_\sharp \Theta=\Gamma$ and $(\pi_1,\pi_2,\pi_3)_\sharp \Theta=\Lambda$.
\newline
We integrate the inequality \eqref{eq:majorationdistancetore} against $\Theta$ to obtain 
\begin{equation}
\begin{split}
\int_{\T^d\times\T^d\times\R^d\times\R^d}|(\exp(x,v)-y|_{\T^d}^2& d\Theta(x,y,z,v)\le\\
&\int_{\T^d\times\T^d\times\R^d\times\R^d}|v|^2+|x-y|^2+ 2v\cdot \zeta(x,y) d\Theta(x,y,z,v)
\end{split}
\end{equation}
Computing each terms individually gives the result.
\end{proof}
\begin{definition}
\label{def:sub/supersol}
A usc (resp. lsc) function $\mathcal U:[0,T]\times \mesTdeux\to\R$ is said to be a viscosity sub-solution (resp. super-solution) of \eqref{eq:PDE} if for any $t\in(0,T],\, \mu\in\mesTdeux$, $(\theta,\psi)\in \partial^+ \mathcal U(t,\mu)$, (resp. $(\theta,\psi)\in \partial^-\mathcal U(t,\mu)$), it holds that 
\begin{equation}
\label{eq:sub/supersol}
\theta-\int_{\T^d\times\R^d}\left(b(x,\mu).z+\cfrac{1}{2}\,|z|^2\right)\psi_x(dz)\mu(dx)\le 0 \,\, (\text{resp. }\ge 0).
\end{equation}
A continuous function $\mathcal U:[0,T]\times\mesTdeux\to\R$ is said to be a viscosity solution of \eqref{eq:PDE} if it is both a sub-solution and a super-solution.
\newline
We say that $\mathcal U_0:\mesT\to\R $ is the initial condition of $\mathcal U$ if $\mathcal U(0,.)=\mathcal U_0$.
\end{definition}
\subsubsection{Comparison principle and uniqueness of the viscosity solution of the limit PDE}
In this part we establish the main result of this section which is the comparison principle for solutions of \eqref{eq:PDE} which in particular implies the uniqueness of viscosity solution of \eqref{eq:PDE}.
\begin{theorem}
Let $\mathcal U$ be a sub-solution of \eqref{eq:PDE} and $\mathcal V$ be a super-solution of \eqref{eq:PDE} such that $\mathcal U(0,.)\le\mathcal V(0,.)$. Then $\mathcal U(t,.)\le \mathcal V(t,.)$ for all $t\ge 0$.
\end{theorem}
\begin{proof}
Let $\delta>0$ and introduce $\mathcal U_\delta(t,\mu):=\mathcal U(t,\mu)-\delta t$. The function $U_\delta$ is a $\delta$ strict sub-solution in the sense that in the inequality \eqref{eq:sub/supersol} the second term is replaced by $-\delta$.
Assume that we succeed to prove the comparison principle for all $\delta>0$ strict sub-solution. 
Then, we have that for all $\delta>0$, for all $t\ge 0$, $\mathcal U_\delta(t,.)\le \mathcal V(t,.)$. Letting $\delta$ goes to 0 yields the result. 
So, starting from now we assume that $\mathcal U$ is a $\delta>0$ strict sub-solution of \eqref{eq:PDE}.
\newline
As usual we use the technique of doubling the variables and we consider for $\varepsilon>0$ the function $$\phi_\varepsilon(t,s,\mu,\nu):=\mathcal U(t,\mu)-\mathcal V(s,\nu)-\cfrac{1}{2\varepsilon}\,((t-s)^2+W_2^2(\mu,\nu)).$$
Since $[0,T]^2\times\mesT^2$ is compact $M_\varepsilon:=\max_{[0,T]^2\times\mesT^2} \phi_\varepsilon$ exists and we can consider a point of maximum $(t_\varepsilon,s_\varepsilon,\mu_\varepsilon,\nu_\varepsilon)$. 
As usual the penalization goes to 0 \begin{equation}
\label{eq:penalizationgoes0}
\cfrac{(t_\varepsilon-s_\varepsilon)^2}{ \varepsilon}\underset{\varepsilon\to 0}{\longrightarrow}0 \text{  and  } \cfrac{W_2^2(\mu_\varepsilon,\nu_\varepsilon)}{ \varepsilon}\underset{\varepsilon\to 0}{\longrightarrow}0.
\end{equation}
Up to extraction we can assume that $t_\varepsilon\underset{\varepsilon\to 0}{\longrightarrow}t^*\in[0,T]$ (and so $s_\varepsilon\underset{\varepsilon\to 0}{\longrightarrow}t^*$) and $\mu_\varepsilon\underset{\varepsilon\to 0}{\longrightarrow}\mu^*\in\mesT$ (and so $\nu_\varepsilon\underset{\varepsilon\to 0}{\longrightarrow}\mu^*$).
\newline
Moreover, we as usual have 
\begin{align*}
M:=\max_{[0,T]\times\mesT} (\mathcal U-\mathcal V)&\ge \mathcal U(t^*,\mu^*)-\mathcal V(t^*,\mu^*)\\
&\ge \underset{\varepsilon\to 0}{\limsup}\left(\mathcal U(t_\varepsilon,\mu_\varepsilon)-\mathcal V(s_\varepsilon,\nu_\varepsilon)\right)\\
&\underset{\eqref{eq:penalizationgoes0}}{\ge}\underset{\varepsilon\to 0}{\limsup}\,M_\varepsilon\ge\underset{\varepsilon\to 0}{\liminf} \,M_\varepsilon \ge M,
\end{align*}
where the last inequality comes from $M_\varepsilon\ge \sup_{t\in[0,T],\mu\in\mesT}\phi_\varepsilon(t,t,\mu,\mu)=M$. 
\newline
Hence, all the previous inequalities are equalities and in particular we obtain that $$ \underset{\varepsilon\to 0}{\lim}M_\varepsilon=M=\mathcal U(t^*,\mu^*)-\mathcal V(t^*,\mu^*).$$
\newline
Let us distinguish two cases.
\newline
$\bullet$ $t^*=0$. In this case, we have $M=\max_{[0,T]\times\mesT} (\mathcal U-\mathcal V)=\mathcal U(0,\mu^*)-\mathcal V(0,\mu^*)\le 0$ by hypothesis giving the result. 
\newline
$\bullet$ $t^*>0$. For $\varepsilon>0$ small enough we have $t_\varepsilon>0$ and $s_\varepsilon>0$. 
\newline
Let $\gamma_\varepsilon$ be an optimal coupling between $\mu_\varepsilon$ and $\nu_\varepsilon$ and let $\zeta$ as in Proposition \ref{prop:subdiffW}. Using that $(t_\varepsilon,s_\varepsilon, \mu_\varepsilon,\nu_\varepsilon)$ is a point of maximum of $\phi_\varepsilon$ and Proposition \ref{prop:subdiffW} we get $$\left(\cfrac{t_\varepsilon-s_\varepsilon}{\varepsilon},(\pi_1,\varepsilon^{-1}\zeta)_\sharp\gamma_\varepsilon\right)\in\mathcal U^+(s_\varepsilon,\nu_\varepsilon),$$
and  $$\left(\cfrac{t_\varepsilon-s_\varepsilon}{\varepsilon},(\pi_2,\varepsilon^{-1}\zeta)_\sharp\gamma_\varepsilon\right)\in\mathcal V^-(t_\varepsilon,\mu_\varepsilon).$$
The sub and super solution properties of $\mathcal U$ and $\mathcal V$ yield
\begin{equation}
\label{eq:principecomp}
\left\{
\begin{aligned}
\cfrac{t_\varepsilon-s_\varepsilon}{\varepsilon}-\int_{(\T^d)^2} \left(b(x,\mu_\varepsilon)\cdot\frac{\zeta(x,y)}{\varepsilon}+\frac{1}{2}\left|\frac{\zeta(x,y)}{\varepsilon}\right|^2\right)\gamma_\varepsilon(dx,dy)&\le-\delta\\
\cfrac{t_\varepsilon-s_\varepsilon}{\varepsilon}-\int_{(\T^d)^2} \left(b(y,\nu_\varepsilon)\cdot\frac{\zeta(x,y)}{\varepsilon}+\frac{1}{2}\left|\frac{\zeta(x,y)}{\varepsilon}\right|^2\right)\gamma_\varepsilon(dx,dy)&\ge 0
\end{aligned}
\right.
\end{equation}
We subtract the two inequalities of \eqref{eq:principecomp} to obtain 
\begin{equation}
\label{eq:principecomp2}
\int_{(\T^d)^2}\left(\left(b(x,\mu_\varepsilon)-b(y,\nu_\varepsilon)\right)\cdot\frac{\zeta(x,y)}{\varepsilon}\right)\gamma_\varepsilon(dx,dy)\ge\delta.
\end{equation}
Since $b$ is a Lipschitz function with constant $K$ we get 
\begin{equation}
\label{eq:principecomp3}
\begin{split}
\delta &\le \int_{(\T^d)^2}\cfrac{|\zeta(x,y)|}{\varepsilon} \,\left|b(x,\mu_\varepsilon)-b(y,\nu_\varepsilon)\right|\gamma_\varepsilon(dx,dy)\\
&\le K\int_{(\T^d)^2}\cfrac{|\zeta(x,y)|}{\varepsilon} \,\left(|x-y|+W_2(\mu_\varepsilon,\nu_\varepsilon)\right)\gamma_\varepsilon(dx,dy)\\
&\le K \left[\cfrac{1}{\varepsilon}\int_{(\T^d)^2}|x-y|^2\gamma_\varepsilon(dx,dy)+\cfrac{W_2(\mu_\varepsilon,\nu_\varepsilon)}{\varepsilon}\int_{(\T^d)^2}|x-y|\gamma_\varepsilon(dx,dy)\right]\\
&\le 2K \cfrac{W_2^2(\mu_\varepsilon,\nu_\varepsilon)}{\varepsilon},
\end{split}
\end{equation}
where the last inequality comes from the Cauchy-Schwarz inequality and the fact that $\gamma_\varepsilon$ is an optimal plan for $W_2(\mu_\varepsilon,\nu_\varepsilon)$.
Passing to the limit $\varepsilon$ goes to $0$, \eqref{eq:penalizationgoes0} gives a contradiction since $\delta>0$.
\end{proof}

\begin{corollary}
\label{coro:uniquenessofviscositysolution}
Given $\mathcal U_0:\mesTdeux\to\R$, there exists at most one continuous viscosity solution of \eqref{eq:PDE} with initial condition $\mathcal U_0$. 
\end{corollary}

\subsection{Convergence toward the limit PDE}

\subsubsection{Compactness}
We first recall the Arzela-Ascoli theorem on the space of probability measures. This theorem can easily be obtained as an extension procedure that preserves the modulus of continuity and then by applying the standard Arzela-Ascoli theorem (see for instance \cite{bertucci2026analysis}).
\begin{theorem}
\label{thm:arzela-ascoli}
Let $(U_N)_{N\ge 1}$ be a sequence of functions such that $U_N: [0,T]\times (\T^d)^N$ is symmetric with respect to its second argument.
We assume that there exists a modulus of continuity $\omega$ and a constant $M>0$ such that 
\begin{enumerate}
\item{For all $N\ge 1$, for all $x,y\in (\T^d)^N$ and for all $s,t\in[0,T]$ we have $$|U_N(s,x)-U_N(t,y)|\le \omega\left(|s-t|+W_2(\mu_x^N,\mu_y^N)\right),$$}
\item{For all $N\ge 1$ $$\sup_{t\in[0,T],\,x\in (\T^d)^N} |U_N(t,x)|\le M.$$}
\end{enumerate}
Then, up to extraction, there exists $\mathcal U:[0,T]\times \mesTdeux\to \R$ with modulus of continuity $\omega$ such that $$\sup_{t\in[0,T],\,x\in (\T^d)^N} |U_N(t,x)-\mathcal U(t,\mu_x^N)|\underset{N\to+\infty}{\longrightarrow} 0.$$
\end{theorem}
We can now prove the compactness of $(\Lambda_N^f)_{N\ge 1}$.
\begin{proposition}
\label{prop:compactness}
There exists a constant $C>0$ such that for all $x,y\in (\T^d)^N$ and $t,s\in [0,T]$\begin{align*}
\left\{
\begin{aligned}
\sup_{t'\in [0,T]}\left|\Lambda_N^f(t',x)-\Lambda_N^f(t',y)\right|&\le C W_2(\mu_x^N,\mu_y^N)\\
\sup_{x'\in(\T^d)^N}|\Lambda_N^f(t,x')-\Lambda_N^f(s,x')|&\le C(|t-s|+\sqrt{|t-s|})
\end{aligned}
\right.
\end{align*}
Hence, up to extraction, there exists $\mathcal U^f:[0,T]\times \mesT\to \R$ such that $$\sup_{t\in[0,T],\,x\in (\T^d)^N} |\Lambda_N^f(t,x)-\mathcal U^f(t,\mu_x^N)|\underset{N\to+\infty}{\longrightarrow} 0.$$
\end{proposition}
These type of estimates are classical in optimal control theory in $\R^d$ \cite{fleming2006controlled} and on the space of probability measures \cite{bertucci2025tangent,bertucci2026analysis}.
We give a proof of this result in Appendix \ref{Appendix:D} to lighten the presentation.

\subsubsection{Convergence}
In this section we establish the most technical result of this paper. To do this we shall use a standard technique when studying PDE on the space of the probability measures which is to lift the problem. The idea is to lift any functions $h:\mesTdeux\to \R$ defined on $\mesTdeux$ on an Hilbert space like $L^2(\Omega',\mathcal F',\mP')$ by defining a lifted version $\tilde h(X)=h(\text{Law}(X))$. This technique has been used for studying Hamilton-Jacobi equations on the space of probability measures since the PDE now become a PDE set on an Hilbert space. On the topic of Hamilton-Jacobi equations in infinite dimension see for instance \cite{crandall1985hamilton,gangbo2019differentiability,gangbo2008hamilton,swiech2009pde}.
Let us mention that here we shall use this lift because it will be more explicit and convenient to construct random variables on a probability space than constructing transport plans.
\newline
In this section only to lighten notation, we write $u_N:=\Lambda_N^f$.
 
\paragraph{The lift strategy} We shall work in the following space $$\mathcal H=L^2(((0,1),\mathcal B(0,1),\leb),\T^d)$$ with $\mathcal B(0,1)$ the Borel sigma-algebra of $(0,1)$, on which our random variables will be defined. 
\newline
To lighten notation instead of $((0,1),\mathcal B(0,1),\leb)$ we shall just write $(0,1)$ when there is no ambiguity.
Given an element $X\in \mathcal H$, we write $\text{ Law}(X)$ for the law of $X$ which is a probability measure in $\mesTdeux$. To avoid confusions, in this section only, we forget the initial probability space $(\Omega,\mathcal F,\mP)$ on which we worked and we use the following notation: for $\mathbb X=\R^d$ or $\mathbb X=\T^d$ and $M,N\in  
L^2(((0,1),\mathcal B(0,1),\leb), \mathbb X)$ we write  $\E(M)=\int_0^1 M(x)dx$ and $\|M-N\|_2^2=\E \left[d_{\mathbb  X}(M,N)^2\right]$. 
\newline
We introduce the sub-differential and the super-differential of a function defined on $\mathcal H$.
\begin{definition}
Let $\tilde U:\mathcal H\to \R$ be a usc function and $X\in \mathcal H$. We define the super-differential of $\tilde U$ in $X$ as the set of $Z \in L^2((0,1),\mathcal B(0,1),\leb),\R^d)$ such that there exists a modulus of continuity $\omega$ such that for every $V\in L^2((0,1),\mathcal B(0,1),\leb),\R^d)$ we have 
\begin{equation}
\label{eq:subdiffhilbert}
\tilde U(X+V \text{ modulo } \Z^d)\le U(X)+\E[Z\cdot V]+\|V\|_2\, \omega(\|V\|_2).
\end{equation}
We denote by $\partial^+ \tilde U(X)$ this set.
\newline 
We similarly define the sub-differential of $\tilde U$ a lsc function at $X$ where the inequality \eqref{eq:subdiffhilbert} is reversed and it will be denoted by $\partial^-\tilde U(X)$.
\end{definition}
\begin{remark}
As said in Remark \ref{rem:extensiontemps} for a function $\tilde U:[0,T]\times \mathcal H\to \R$ we can define the immediate analogues $\partial^+\tilde U(t,X)$ and $\partial^-\tilde U(t,X)$ for $(t,X)\in [0,T]\times \mathcal H$.
\end{remark}
We can now a define a notion of solutions of \eqref{eq:PDE} using the lift.
\begin{definition}
\label{def:sub/supersolhilbert}
A usc (resp.lsc) function $\tilde U:[0,T]\times \mathcal H \to\R$ is said to be a viscosity sub-solution (resp. super-solution) of \eqref{eq:PDE} if for any $t\in(0,T],\, X\in \mathcal H$, $(\theta,Z)\in \partial^+ \tilde U(t,X)$, (resp. $(\theta,Z)\in \partial^-\tilde  U(t,X)$), it holds that 
\begin{equation}
\label{eq:sub/supersolhilbert}
\theta-\E[b(X,\mathrm{ Law}(X))\cdot Z]-\frac{1}{2}\E[|Z|^2]\le 0 \,\, (\text{resp. }\ge 0).
\end{equation}
A continuous function $\tilde U:[0,T]\times\mathcal H\to\R$ is said to be a viscosity solution of \eqref{eq:PDE} if it is both a sub-solution and a super-solution.
\newline
We say that $\tilde U_0:\mesT\to\R $ is the initial condition of $\tilde U$ if $\tilde U(0,.)=\tilde U_0$.
\end{definition}
By analogy of these definitions with the one presented for functions defined on $\mesTdeux$ we obtained the following result.
\begin{proposition}
\label{coro:liftsolimpliqueviscosol}
Given $\mathcal U: [0,T]\times \mesTdeux\to \R$ a continuous function we define its lift $\tilde U [0,T]\times \mathcal H \to \R$ by $\tilde U(t,X)=\mathcal U(t,\mathrm{ Law}(X))$. If $\tilde U$ is a viscosity solution of \eqref{eq:PDE} then $\mathcal U$ is a viscosity solution of \eqref{eq:PDE}.
\end{proposition}
The proof is given in \ref{Appendix:C} for the sake of completeness of the reader.

\paragraph{Convergence using the lift}
In this section we prove the convergence of $(u_N=\Lambda_N^f)_{N\ge 1}$. The strategy that we use is classical and consists of approaching points of the differential of first order of a function by elements of the differential of second order. 
\begin{theorem}
\label{thm:convergence}
The sequence of functions $(u_N)_{N\ge 1}$ converges uniformly toward $\mathcal U^f$ the unique viscosity solution of \eqref{eq:PDE} with initial condition $f$.
More explicitly we have $$\sup_{t\in[0,T],x\in (\T^d)^N}|u_N(t,x)-\mathcal U^f(t,\mu_x^N)|\underset{N\to+\infty}{\longrightarrow} 0.$$
\end{theorem}
\begin{proof}
By Proposition \ref{prop:compactness}, the sequence of functions $(u_N)_{N\ge 1}$ is relatively compact. Hence it remains to prove that there is a unique limit point.
We suppose to lighten notation that $(u_N)_{N\ge1}$ converges uniformly toward a continuous function $\mathcal V:[0,T]\times \mesTdeux\to \R$. Our goal to obtain the convergence of the full sequence is to characterize $\mathcal V$ by showing that $\mathcal V$ is the unique viscosity solution of \eqref{eq:PDE} with initial condition $f$. To do this, we will use Proposition \ref{coro:liftsolimpliqueviscosol} so that it remains to prove that its lift $\tilde V:[0,T]\times \mathcal H\to \R$ defined by $\tilde V(t,X)=\mathcal V(t,\mathrm{ Law}(X))$ is a viscosity solution. Once this result obtained, the uniqueness of viscosity solutions of \eqref{eq:PDE} from Corollary \ref{coro:uniquenessofviscositysolution} concludes the proof. 
\newline
We only show the sub-solution property since the super-solution property can be done similarly.
\newline
\underline{Step 1: Discretization and notation.}
\newline We fix $(t_0,X_0)\in (0,T]\times L^2((0,1),\T^d)$ and $(\theta,Z)\in\partial^+\tilde V(t_0,X_0)$. In this proof, we will denote by $C$ any constant that only depends on $d,T,\|X_0\|_2,\|Z\|_2$ and we consider $\omega$ the modulus of continuity associated to $(\theta,Z)\in\partial^+\tilde V(t_0,X_0)$.
\newline
We introduce $A_{i,N}=\mathds 1_{((i-1)/N,i/N)}$ to discretize $(0,1)$. By a density argument, we can find $(z^N:=(z_1^N,...z_N^N))_{ N\ge 1}$ with $z^N\in (\R^d)^N$ such that $$Z^N:=\sum_{i=1}^N z_i^N A_{i,N}\overunderset{L^2}{N\to+\infty}{\longrightarrow} Z.$$ We similarly construct  $(x^N:=(x_1^N,...x_N^N))_{ N\ge 1}$ with $x^N\in (\T^d)^N$ such that $$X^N:=\sum_{i=1}^N x_i^N A_{i,N}\overunderset{L^2}{N\to+\infty}{\longrightarrow} X_0.$$ Let us notice that the law of $X^N$ is $\mu^N_{X^N}$ and the law of $Z^N$ is $\mu^N_{Z^N}$.
\newline
Moreover, we introduce an error term
\begin{align*}
\delta_N:=\|X^N-X_0\|&_2+\|Z^N-Z\|_2+\\&\sup_{t\in[0,T],x\in (\T^d)^N}|u_N(t,x)-\mathcal \mathcal V(t,\mu_x^N)|+|\tilde V(t_0,X_0)-\tilde V(t_0,X^N)|
\end{align*} which is a quantity that goes to 0 by construction and continuity of $\mathcal V$.
\newline
Given a a vector $v\in\T^d$ or in $\R^d$, we recall that its sinus is defined by $$\sin(v)=(2\pi)^{-1}(\sin(2\pi v_1),...,\sin(2\pi v_d))\in \R^d.$$
We also recall that $\zeta:\T^d\times\T^d\to \R^d$ is defined by $\zeta(x,y)=x-y-k(x,y)$ where $k(x,y)\in\Z^d$ is chosen such that $x-y-k(x,y)\in[-2^{-1},2^{-1})^d$ (and so $|\zeta(x,y)|=d_{\T^d}(x,y)$).
In particular for all $(x,y)\in(\T^d)^2$ we have $\left|\sin\left(\frac{y-x}{2}\right)\right|=\left|\sin\left(\frac{\zeta(x,y)}{2}\right)\right|$ and using that $\zeta(x,y)\in [-2^{-1},2^{-1})^d$ and the inequality $$\frac{2}{\pi}|x|\le |\sin(x)|\le |x|$$ for $x\in [0,\pi/2]$ we obtain that there exists a constant $C>0$ such that for all $(x,y)\in(\T^d)^2$ 
\begin{equation}
\label{eq:sinusborne1}
C^{-1} d_{\T^d}(x,y)\le \left|\sin\left(\frac{y-x}{2}\right)\right|\le C d_{\T^d}(x,y).
\end{equation}
Let us also mention another inequality that shall be used later based on the existence of $C>0$ such that for all $r<2^{-1}$ $$\left|\frac{1}{2\pi}\sin(2\pi r)-r\right|\le C |r|^3.$$
This inequality implies that there exists a constant $C>0$ such that for all $(x,y)\in(\T^d)^2$ we get 
\begin{equation}
\label{eq:sinusborne2}
|\sin(x-y)-\zeta(x,y)|\le C |\zeta(x,y)|^3.
\end{equation}
\underline{Step 2: Choice of test functions.}
\newline
We define the following smooth test functions $\phi_N :[0,T]\times (\T^d)^N\to \R$: \begin{equation}
\begin{split}
\phi_N(t,y^N:=(y_1^N,...,y_N^N)):=\theta(t-t_0)+\cfrac{1}{N}\sum_{i=1}^N z_i^N\cdot &\,\sin(y_i^N-x_i^N)+\\
&\cfrac{|t-t_0|^2}{2\gamma_N}+\cfrac{1}{2N\gamma_N}\sum_{i=1}^N\left|\sin\left(\frac{y_i^N-x_i^N}{2}\right)\right|^2,
\end{split}
\end{equation}
where $\gamma_N$ is a sequence that goes to 0 slow enough and that shall be chosen later on. 
For $y^N\in(\T^d)^N$ we write $Y^N=\sum_{i=1}^N y_i^N A_{i,N}\in L^2((0,1),\T^d)$ and for $P,Q$ in $L^2((0,1),\T^d)$ or in $L^2((0,1),\R^d)$ $$S(P,Q)=\sin(P-Q)\in L^2((0,1),\R^d)$$ and $$D(P,Q)=\left\|\sin\left(\frac{P-Q}{2}\right)\right\|_2.$$ We can rewrite the functions $\phi_N$ as follows
$$\phi_N(t,y^N)=\theta(t-t_0)+\E[Z^N.\,S(Y^N,X^N)]+\cfrac{|t-t_0|^2}{2\gamma_N}+\cfrac{	D(X^N,Y^N)^2}{2\gamma_N}.$$
\newline
We now take a point of maxima $(t^N,y^N)\in [0,T]\times (\T^d)^N$ of $u_N-\phi_N$ on $[0,T]\times (\T^d)^N$.
\newline
\underline{Step 3: For $N$ large $(t_N,y^N)$ is near $(t_0,X_0)$.}
\newline
The goal of this part is to prove that for a good choice of $(\gamma_N)_{N\ge 1}$ we have $$\cfrac{|t_N-t_0|}{\gamma_N}+\cfrac{D(X^N,Y^N)}{\gamma_N}\underset{N\to+\infty}{\longrightarrow} 0.$$
Using that $(t_N,y^N)$ is a point of maximum of $u_N-\phi_N$ we have 
$$u_N(t_N,y^N)-\phi_N(t_N,y^N)\ge u_N(t_0,x^N)-\phi_N(t_0,x^N)=u_N(t_0,x^N).$$
This yields 
\begin{equation}
\label{eq:boundfunctiontest}
\begin{split}
\phi_N(t_N,y^N)&=\theta(t_N-t_0)+\cfrac{|t_N-t_0|^2}{2\gamma_N}+\E[Z^N\cdot\,S(Y^N,X^N)]+\cfrac{D(X^N,Y^N)^2}{2\gamma_N}\\
&\le u_N(t_N,y^N)-u_N(t_0,x^N)\\
&\le 2\delta_N+\tilde V(t_N,Y^N)-\tilde V(t_0,X^N)\\
&\le 3\delta_N+\tilde V(t_N,Y^N)-\tilde V(t_0,X_0).
\end{split}
\end{equation}
To bound the right term in \eqref{eq:boundfunctiontest} we now use the fact that $(\theta,Z)\in \partial^+ \tilde V(t_0,X_0)$ to obtain that choosing $V^N:=\zeta(Y^N,X_0)$ as a velocity 
\begin{equation}
\label{eq:boundsousdif}
\begin{split}
&\tilde V(t_N,Y^N)-\tilde V(t_0,X_0)\le\\
&\theta(t_N-t_0)+\E[Z\cdot V^N]+\left(\sqrt{|t_N-t_0|^2+D(X_0,Y^N)^2}\right)\omega\left(\sqrt{|t_N-t_0|^2+D(X_0,Y^N)^2}\right),
\end{split}
\end{equation}
where we used that $$\|V^N\|_2^2=\E [|\zeta(Y^N,X_0)|^2]=\E [d_{\T^d}(Y_N,X_0)^2]\le C D(X_0,Y^N)^2.$$
Moreover using \eqref{eq:sinusborne1} there exists a constant $C$ such that $$D(X_0,Y^N)^2\le C(D(X_0,X^N)^2+D(X^N,Y^N)^2)
\le C(D(X^N,Y^N)^2+\delta_N^2)$$ which gives that (up to changing the modulus of continuity $\omega$ in $\omega'(t)=C\omega(Ct)$)
\begin{equation}
\label{eq:boundsousdif2}
\begin{split}
\tilde V(t_N,Y^N)-&\tilde V(t_0,X_0)\le\theta(t_N-t_0)+\E[Z\cdot V^N]+\\
&\left(\sqrt{|t_N-t_0|^2+D(X^N,Y^N)^2}+\delta_N\right)\omega\left(\sqrt{|t_N-t_0|^2+D(X^N,Y^N)^2}+\delta_N\right).
\end{split}
\end{equation}
Introducing $h_N^2:=|t_N-t_0|^2+D(X^N,Y^N)^2$ and combining \eqref{eq:boundfunctiontest} and \eqref{eq:boundsousdif2} give
\begin{equation} 
\label{eq:boundfunctiontest2}
\begin{split}
\cfrac{h_N^2}{2\gamma_N}&\le 3\delta_N+\E[Z\cdot V^N]-\E[Z^N\cdot\,S(Y^N,X^N)]+(h_N+\delta_N)\omega\left(h_N+\delta_N\right)\\
&=: 3\delta_N+\Delta_N+(h_N+\delta_N)\omega\left(h_N+\delta_N\right).
\end{split}
\end{equation}
The problem here is that $\Delta_N$ depends on $Y^N$ and thus on $\gamma_N$. Hence it is not clear how to choose $\gamma_N$ to make $h_N/\gamma_N$ goes to 0 directly using \eqref{eq:boundfunctiontest2}. We have to bound $\Delta_N$ in order to proceed:
\begin{equation}
\label{eq:bounddedelta}
\begin{split}
\Delta_N&=\E[Z\cdot V^N]-\E[Z^N\cdot\,S(Y^N,X^N)]\\
&=\E[Z\cdot (V^N-S(Y^N,X^N)]+\E[(Z-Z^N)\cdot\,S(Y^N,X^N)]\\
&\le \E[Z\cdot (V^N-S(Y^N,X^N)]+\|Z-Z^N\|_2\|S(Y^N,X^N)\|_2\\
&\le \E[Z\cdot (V^N-S(Y^N,X^N)]+C\|Z-Z^N\|_2,
\end{split}
\end{equation}
where we used the Cauchy-Schwarz inequality and the fact that the sinus is bounded.
We get  
\begin{equation}
\label{eq:bounddedelta2}
\begin{split}
\Delta_N&\le \E[Z\cdot (V^N-S(Y^N,X^N)]+C\|Z-Z_N\|_2\\
&\le \E[Z\cdot (V^N-S(Y^N,X_0)]+\E[Z\cdot (S(Y^N,X_0)-S(Y^N,X^N)]+C\|Z-Z^N\|_2 \\
&\le  \E[Z\cdot (V^N-S(Y^N,X_0)]+\|Z\|_2\|S(Y^N,X_0)-S(Y^N,X^N)\|_2+C\|Z-Z^N\|_2\\
&\le \E[Z\cdot (V^N-S(Y^N,X_0)]+C\|X^N-X_0\|_2+C\|Z-Z^N\|_2,
\end{split}
\end{equation}
where we used the Cauchy-Schwarz inequality and that the sinus is a Lipschitz function.
\newline
It remains to bound $\E[Z\cdot (V^N-S(Y^N,X_0)]$ using \eqref{eq:sinusborne2}:
\begin{align*}
E[Z\cdot (\zeta(Y^N,X_0)-S(Y^N,X_0)]&\le \E[|Z|\,|\zeta(Y^N,X_0)-S(Y^N,X_0)|]\\
&\le C\E[|Z|\,|\zeta(Y^N,X_0)|^3]\\
&\le C \E[|Z|\,|\zeta(Y^N,X_0)|^2]
\end{align*}
using that $\zeta$ is bounded. If $Z$ were bounded we could conclude but $Z$ is only in $L^2$ so we shall truncate it.
\newline
For $K>0$ we define $Z_K=Z \mathds{1}_{|Z|\le K}$ and we write 
\begin{align*}
\E[Z\cdot (\zeta(Y^N,X_0)-S(Y^N,X_0)]&\le C(\E[|Z-Z_K|\,|\zeta(Y^N,X_0)|^2]+\E[|Z_K|\,|\zeta(Y^N,X_0)|^2])\\
&\le C(\|Z-Z_K\|_2\|\,|\zeta(Y^N,X_0)|^2\|_2+K\E[|\zeta(Y^N,X_0)|^2])\\
&\le C(\|Z-Z_K\|_2(h_N+\delta_N)+K (h_N^2+\delta_N^2))
\end{align*}
In summary we obtain for every $K>0$
\begin{equation} 
\label{eq:boundfunctiontest3}
\begin{split}
\cfrac{h_N^2}{2\gamma_N}&\le C\delta_N+C(\|Z-Z_K\|_2(h_N+\delta_N)+K (h_N^2+\delta_N^2))+(h_N+\delta_N)\omega\left(h_N+\delta_N\right).
\end{split}
\end{equation}
First fixing a $K>0$, using that $(h_N)_{N\ge 1}$ is bounded and using that $(\gamma_N)_{N\ge 1}$ goes to 0, the inequality \eqref{eq:boundfunctiontest3} implies that $(h_N)_{N\ge 1}$ goes to 0. 
We set $H_N:=h_N/\gamma_N$ and $a_N:=\delta_N/\gamma_N$. Dividing \eqref{eq:boundfunctiontest3} by $\gamma_N$ yields
\begin{equation} 
\label{eq:boundfunctiontest4}
\begin{split}
\cfrac{H_N^2}{2}&\le C a_N+C(\|Z-Z_K\|_2(H_N+a_N)+K\gamma_N (H_N^2+a_N^2))+(H_N+a_N)\omega\left(h_N+\delta_N\right).
\end{split}
\end{equation}
This gives 
\begin{equation} 
\label{eq:boundfunctiontest5}
\begin{split}
\left(\cfrac{1}{2}-CK\gamma_N \right)H_N^2\le C a_N+C\|Z-Z_K\|_2 a_N+&CK\gamma_N a_N^2+a_N\omega\left(h_N+\delta_N\right)+\\
&H_N\left(C\|Z-Z_K\|_2+\omega\left(h_N+\delta_N\right)\right).
\end{split}
\end{equation}
We now choose $\gamma_N\ge \sqrt{\delta_N}$ which goes to 0 so that $(a_N)_{N\ge 1}$ goes to 0 and $K_N:=\gamma_N^{-1/2}$ goes to $+\infty$. So with this choice $(K_N\gamma_N)_{N\ge 1}$ goes to 0 and so for $N$ large enough we obtain
\begin{equation} 
\label{eq:boundfunctiontest6}
\begin{split}
\cfrac{H_N^2}{4}\le C a_N+C\|Z-Z_{K_N}\|_2 a_N+&CK_N\gamma_N a_N^2+a_N\omega\left(h_N+\delta_N\right)+\\
&H_N\left(C\|Z-Z_{K_N}\|_2+\omega\left(h_N+\delta_N\right)\right).
\end{split}
\end{equation}
Using Young's inequality for the right hand term yields 
\begin{equation} 
\label{eq:boundfunctiontest7}
\begin{split}
\cfrac{H_N^2}{8}\le C a_N+C\|Z-Z_{K_N}\|_2 a_N+&C\sqrt{\gamma_N} a_N^2+a_N\omega\left(h_N+\delta_N\right)+\\
&2\left(C\|Z-Z_{K_N}\|_2+\omega\left(h_N+\delta_N\right)\right)^2.
\end{split}
\end{equation}
Passing to the limit gives that $(H_N)_{N\ge 1}$ goes to 0 which yields 
\begin{equation}
\label{eq:penagoesto 0bis}
\cfrac{|t_N-t_0|}{\gamma_N}+\cfrac{D(X^N,Y^N)}{\gamma_N}\underset{N\to+\infty}{\longrightarrow} 0.
\end{equation}
Moreover, since $(D(X^N,X_0))_{N\ge 1}$ goes to 0 because $(X^N)_{N\ge 1}$ because to $X_0$ in $L^2$ we can choose $\gamma_N\ge \sqrt{D(X^N,X_0)}$ that goes to 0 such that $$\cfrac{D(X_0,Y^N)}{\gamma_N}\underset{N\to+\infty}{\longrightarrow} 0$$
\underline{Step 4: Convergence.}
\newline
Since $u_N$ is a smooth solution of \eqref{eq:edphofcole} it is in particular a viscosity solution and so using that $(t_N,y^N)$ is a point of maximum of $u_N-\phi_N$ we get 
\begin{equation}
\label{eq:passagelim1}
\begin{split}
\partial_t&\phi_N(t_N,y^N)\le\\
&\dis\sum_{i=1}^Nb(y_i^N,\mu_{y^N}^N).\nabla_{y_i}\phi_N(t_N,y^N)+\cfrac{\sigma_N^2}{2}\dis\sum_{i=1}^N\Delta_{y_i}\phi_N(t_N,y^N) +\cfrac{N}{2}\dis\sum_{i=1}^N|\nabla_{y_i}\phi_N(t_N,y^N)|^2.
\end{split}
\end{equation}
As in Section \ref{subsection:exptightness}, given $z=(z_1,...,z_d)\in \T^d$ we write $\cos(z)$ the diagonal matrix of size $d\times d$ with diagonal coefficient given by $(\cos(2\pi z_1),...,\cos(2\pi z_d))$. 
We can more explicitly write \eqref{eq:passagelim1} 
\begin{equation}
\label{eq:passagelim2}
\begin{split}
\theta +\cfrac{t_N-t_0}{\gamma_N}&\le\E\left[b(Y^N,\text{ Law}(Y^N))\cdot\left( \cos(Y^N-X^N) Z^N+\cfrac{S(Y^N,X^N)}{4\gamma_N}\right)\right] +\\
&\cfrac{\sigma_N^2}{2}\sum_{i=1}^N\Delta_{y_i}\phi_N(t_N,y^N) +\cfrac{1}{2}\,\E\left[\left|\cos(Y^N-X^N)Z^N+\cfrac{S(Y^N,X^N)}{4\gamma_N}\right|^2\right].
\end{split}
\end{equation}
About the second order term, using that the sinus and the cosinus are bounded we obtain that $$\left|\sum_{i=1}^N\Delta_{y_i}\phi_N(t_N,y^N)\right|\le C \left(D(Y^N,X^N)+\cfrac{1}{\gamma_N}\right).$$
Hence we obtain
\begin{equation}
\label{eq:passagelim3}
\begin{split}
\theta +\cfrac{t_N-t_0}{\gamma_N}&\le\E\left[b(Y^N,\text{ Law}(Y^N))\cdot \left(\cos(Y^N-X^N)Z^N+\cfrac{S(Y^N,X^N)}{4\gamma_N}\right)\right] +\\
&\cfrac{C\sigma_N^2}{2}\left(D(Y^N,X^N)+\cfrac{1}{\gamma_N}\right)+\cfrac{1}{2}\,\E\left[\left|\cos(Y^N-X^N)Z^N+\cfrac{S(Y^N,X^N)}{4\gamma_N}\right|^2\right].
\end{split}
\end{equation}
Using \eqref{eq:penagoesto 0bis} to obtain that the penalization goes to 0 we get that \begin{align*}
Z^N\overunderset{L^2}{N\to+\infty}{\longrightarrow} Z,&\,\,\,\,\cfrac{S(Y^N,X^N)}{\gamma_N}\le \cfrac {2D_N}{\gamma_N}\overunderset{L^2}{N\to+\infty}{\longrightarrow} 0,\\
\cos(Y^N-X^N)\overunderset{L^2}{N\to+\infty}{\longrightarrow} I_d,&\,\,\,\, b(Y^N,\text{ Law}(Y^N))\overunderset{L^2}{N\to+\infty}{\longrightarrow} b(X_0,\text{ Law}(X_0)).
\end{align*}
Since $(\sigma_N)_{N\ge 1}$ goes to 0 we can choose $\gamma_N\ge \sigma_N$ that goes to 0 and using that the cosinus and $b$ are bounded we can pass to limit in \eqref{eq:passagelim3} to obtain $$\theta\le \E[b(X_0,\text{ Law}(X_0))\cdot Z]+\cfrac{1}{2}\,\E \left[|Z|^2\right],$$
which proves that $\tilde V$ is a sub-solution of \eqref{eq:PDE}. This concludes the proof.
\end{proof}

\subsubsection{Stability with respect to the initial condition}
As explained in Section \ref{section:semigroupprop} we want to prove the convergence of $(\Lambda_N^{g_N})$ toward $\mathcal U^g$ where $g_N:(\T^d)^N\to\R$ is symmetric and converges uniformly toward a Lipschitz function $g:\mesTdeux\to \R$. 
\newline
Recalling that for instance in the case $m=2$ in Section \ref{section:semigroupprop} we are actually interested in the $g_N$ of the form $$g_N(\bullet)=f_1(\bullet)+\Lambda_N^{f_2}(t_2-t_1,\bullet)$$ we can actually assume that there exists a constant of $L$ such that for all $N\ge 1$, $g_N$ is Lipschitz with constant $L$ meaning that
\begin{equation}
\label{hypotheseuniflipschit}
\forall N\ge 1,\,\forall (x,y)\in ((\T^d)^N)^2,\,|g_N(x)-g_N(y)|\le L W_2(\mu_x^N,\mu_y^N).
\end{equation}
Looking carefully at the proof of the compactness in Proposition \ref{prop:compactness}, under the previous assumptions on $g_N$, the exact same proof works. Moreover in the proof of the convergence in Theorem \ref{thm:convergence} we directly prove that a solution of \eqref{eq:edphofcole} which converges uniformly necessarily converges toward a viscosity solution of the PDE \eqref{eq:PDE}. So we can state the following result. 
\begin{theorem}
\label{thm:convergencvestabilité} Let $g_N:(\T^d)^N\to \R$ be a sequence of symmetric functions that converges uniformly toward $g:\mesTdeux\to\R$ and which satisfies \eqref{hypotheseuniflipschit} for a constant $L$. Then $(\Lambda_N^{g_N})_{N\ge1}$ converges uniformly toward $\mathcal U^g$, the unique viscosity solution of \eqref{eq:PDE} with initial condition $g$.
\end{theorem}
As a direct consequence and application of the semigroup property of Proposition \ref{prop:markovproperty}, one can show by induction the following result which concludes the convergence of the transfer function of the $m-$tuple $(X_N(t_1),...,X_N(t_m))$.
\begin{corollary}
\label{Coro:convuniforme}
Given $0\le t_1\le ...\le t_m\le T$ and $f_1,...,f_m\in BL(\mesTdeux)$ we have the following convergence
\begin{equation}
\label{eq:laplacetransfofinidim3}
\sup_{x\in(\T^d)^N}\left|\cfrac{1}{\varepsilon_N}\log\left(\E_x\left[\exp\left(\varepsilon_N \sum_{i=1}^m f_i(X_N(t_i))\right)\right]\right)-\mathcal U^{f_1+\mathcal U^{f_2+...}(t_2-t_1)(\bullet)}(t_1,\mu_x^N)\right|\underset{N\to +\infty}{\longrightarrow} 0
\end{equation}
\end{corollary}
\subsection{Conclusion of the large deviations principle}
We can now state the large deviations theorem and completely prove the main theorem of this paper. 
\begin{theorem}
\label{thm:largedev2}
Assume that the family $(\mathrm{ Law}(X_N(0)))_{N\ge 1}$ satisfies a large deviations principle in $\mesTdeux$ at speed $\varepsilon_N:=N/\sigma_N^2$ with transfer function $\Lambda_0$ defined by $$\Lambda_0(g)= \underset{N\to +\infty}{\lim}\frac{1}{\varepsilon_N}\log\left(\E\left[\exp\left(\varepsilon_N g(X_N(0))\right)\right]\right),$$ for $g\in BL(\mesTdeux)$. 
\newline
Then the sequence of random variables $(X_N)_{N\ge1}$ satisfies a large deviation principle in $C([0,T],\mesTdeux)$ at speed $N/\sigma_N^2$ with good rate function $I$. More precisely, $I$ is given by
\begin{equation}
\label{eq:fonctiondetaux}
I(\mu)=\sup_{m\ge 1}\sup_{0\le t_1\le ...\le t_m\le T} I_{t_1,...,t_m}(\mu(t_1),...,\mu(t_m))
\end{equation}
where $$I_{t_1,...,t_m}(\mu_1,...,\mu_m)=\sup_{f_1,...,f_m\in BL(\mesTdeux)}\left[\sum_{i=1}^m f_i(\mu_i)-\Lambda_0(\mathcal U^{f_1+\mathcal U^{f_2+...}(t_2-t_1)(\bullet)}(t_1,\bullet))\right],$$
where for all $h\in BL(\mesTdeux)$, $\mathcal U^h:[0,T]\times \mesTdeux\to\R$ is the unique viscosity solution of \eqref{eq:PDE} with initial condition $h$.
\end{theorem}
\begin{proof}
By Theorem \ref{thm:exptightness}, the sequence of random variables $(X_N)_{N\ge 1}$ is exponentially tight at speed $\varepsilon_N$. So, by Corollary \ref{coro:prokhorovlargedev} it remains to show that the rate function is uniquely characterized. Using Theorem \ref{thm:ldpmarginalstoproces}, Proposition \ref{prop:markovproperty} and Corollary \ref{Coro:convuniforme} the rate function is necessarily given by the formula \eqref{eq:fonctiondetaux} which proves the uniqueness and concludes the proof.
\end{proof}
\subsection{Explicit formula of the good rate function}
In this section we explain how to obtain an explicit formula of the rate function of $(X_N)_{N\ge 1}$. This formula is based on the representation of the solution of \eqref{eq:PDE} as the value function of an optimal control problem on the space of probability measures. 
\begin{theorem}
\label{thm:formuleexplicite}
Given $f\in BL(\mesTdeux)$, the unique viscosity solution of \eqref{eq:PDE} with initial condition $f$ is given by 
\begin{equation}
\label{eq:solutionexplicit}
\mathcal U^f(t,\mu)=\sup_{(\mu_s,v_s)_{s\in[0,t]}} \left\{f(\mu_t)-\cfrac{1}{2}\int_0^t\|v_s-b(.,\mu_s)\|_{L^2(\mu_s)}^2ds\right\},
\end{equation}
where the supremum is taken on all the $(\mu_s,v_s)_{s\in[0,t]}$ such that $\mu:[0,t]\to\mesTdeux$ is absolutely continuous, $(v_s)_{s\in[0,t]}$ satisfies $$\int_0^t\|v_s\|^2_{L^2(\mu_s)}ds<+\infty,$$ that are  solutions in the sense of distributions of the continuity equation $\partial_s \mu_s+\nabla.(\mu_s v_s)=0$ and with initial condition $\mu_0=\mu$.
\end{theorem}
The complete proof is given in Section 13 of \cite{bertucci2026analysis}.
For readers that are not familiar with optimal control we give a proof in the case when we know that the value function is smooth (which is in general not the case).
\begin{proof}
For $t\ge 0$, $\mu\in\mesTdeux$ we define $$\mathcal V(t,\mu):=\sup_{(\mu_s,v_s)_{s\in[0,t]}} \left\{f(\mu_t)-\cfrac{1}{2}\int_0^t\|v_s-b(.,\mu_s)\|_{L^2(\mu_s)}^2 ds\right\},$$
where the supremum is taken on all the $(\mu_s,v_s)_{s\in[0,t]}$ as in the statement.
We assume that this function is smooth. 
Using the dynamic programming principle we have for all $t,h\ge 0$ and $\mu \in\mesTdeux$
$$ \mathcal V(t+h,\mu):=\sup_{(\mu_s,v_s)_{s\in[0,h]}} \left\{\mathcal V(t,\mu_h)-\cfrac{1}{2}\int_0^h\|v_s-b(.,\mu_s)\|_{L^2(\mu_s)}^2ds\right\}.$$
Doing a Taylor expansion when $h$ goes to $0$ yields 
\begin{equation}
\left\{
\begin{aligned}
\mathcal V(t+h,\mu)&=\mathcal V(t,\mu)+h\partial_t \mathcal V(t,\mu)+o(h)\\
\mathcal V(t,\mu_h)&=\mathcal V(t,\mu)+h\int_{\T^d}D_\mu \mathcal V(t,\mu,x)\cdot v_0\mu(dx)+o(h)\\
\cfrac{1}{2}\int_0^h\|v_s-b(.,\mu_s)\|_{L^2(\mu_s)}^2ds&=\cfrac{h}{2}\,\|v_0-b(.,\mu)\|_{L^2(\mu)}^2 +o(h)
\end{aligned}
\right.
\end{equation}
Dividing by $h$ and letting $h$ goes to 0, we formally obtain that 
\begin{align*}
\partial_t \mathcal V(t,\mu)&=\sup_v\left\{\int_{\T^d}D_\mu \mathcal V(t,\mu,x)\cdot v\mu(dx)-\cfrac{1}{2}\,\int_{\T^d}|v-b(x,\mu)|^2\mu(dx)\right\}\\
&=\int_{\T^d}b(y,\mu)\cdot D_\mu\mathcal V(t,\mu,x)\mu(dx)+\cfrac{1}{2}\int_{\T^d}|D_{\mu}\mathcal V(t,\mu,x)|^2\mu(dx),
\end{align*}
since the supremum is reached for $v(x)=b(x,\mu)+D_\mu\mathcal V(t,\mu,x)$.
\end{proof}
We first give another expression of the rate function of the $m$-tuple, using the explicit formula of the solution of the Hamilton-Jacobi equation \eqref{eq:PDE}.
\begin{lemma}
\label{lemma:rewriteratefunctionmtuple}
For all $0\le t_1\le ...\le t_m$, $(\nu_1,...,\nu_m)\in(\mesTdeux)^m$ we have
\begin{align*}
I_{t_1,...,t_m}(\nu_1,...,\nu_m)=\inf_{\substack{(\mu_s,v_s)_{s\in[0,t_m]}\\\partial_s\mu_s+\nabla\cdot(\mu_sv_s)=0,\\\mu_{t_i}=\nu_i, \,\forall 1\le i\le m}}\left[\cfrac{1}{2}\int_0^{t_m}\|v_s-b(.,\mu_s)\|_{L^2(\mu_s)}^2ds+I_0(\mu(0))\right]. 
\end{align*}
\end{lemma}
\begin{proof}
We recall that the transfer function associated to the LDP of $(X_N(0))_{N\ge 1}$ is given by $\Lambda_0(g)=\sup_{\nu\in\mesTdeux} (g(\nu)-I_0(\nu))$.
\newline
Bootstrapping Theorem \ref{thm:formuleexplicite}, we get that for all $0<t_1\le...\le t_m\le T$ and $\mu\in\mesTdeux$
 \begin{equation}
\label{eq:bootsratpingsol}
\begin{split}
\mathcal U^{f_1+\mathcal U^{f_2+...}(t_2-t_1)(\bullet)}(t_1,\mu)&=\sup_{(\mu_s,v_s)_{s\in[0,t_m]}} \left\{\sum_{i=1}^m f_i(\mu_{t_i})-\cfrac{1}{2}\int_0^{t_m}\|v_s-b(.,\mu_s)\|_{L^2(\mu_s)}^2ds\right\}\\
&=:\mathcal V^m(\mu),
\end{split}
\end{equation}
where the supremum is taken on all the $(\mu_s,v_s)_{s\in[0,t_m]}$ such that $\mu:[0,t_m]\to\mesTdeux$ is absolutely continuous, that are  solutions in the sense of distributions of the continuity equation $\partial_t \mu_t+\nabla.(\mu_t v_t)=0$, with initial condition $\mu_0=\mu$ and that satisfy $$\int_0^{t_m}\|v_s\|^2_{L^2(\mu_s)}ds<+\infty.$$ 
Hence, for all $0< t_1\le ...\le t_m$, $(\nu_1,...,\nu_m)\in(\mesTdeux)^m$
\begin{equation}
\label{eq:inversionI}
\begin{split}
I_{t_1,...,t_m}(\nu_1,...,\nu_m)&=\sup_{f_1,...,f_m\in BL(\mesTdeux)}\left[\sum_{i=1}^m f_i(\nu_i)-\Lambda_0(\mathcal U^{f_1+\mathcal U^{f_2+...}(t_2-t_1)(\bullet)}(t_1,\bullet))\right]\\
&=\sup_{f_1,...,f_m\in BL(\mesTdeux)}\left[\sum_{i=1}^m f_i(\nu_i)-\Lambda_0(\mathcal V^m(\mu))\right]\\
&=\sup_{f_1,...,f_m\in BL(\mesTdeux)}\left[\sum_{i=1}^m f_i(\nu_i)-\sup_{\mu\in\mesTdeux}(\mathcal V^m(\mu)-I_0(\mu))\right]\\
&=\sup_{f_1,...,f_m\in BL(\mesTdeux)}\inf_{\mu\in\mesTdeux}\left[\sum_{i=1}^m f_i(\nu_i)-(\mathcal V^m(\mu)-I_0(\mu))\right]\\
&=\sup_{f_1,...,f_m\in BL(\mesTdeux)}\inf_{\substack{(\mu_s,v_s)_{s\in[0,t_m]}\\\partial_s \mu_s+\nabla.(\mu_s v_s)=0\\
\mu(0)=\mu}}\left[\sum_{i=1}^m f_i(\nu_i)-\sum_{i=1}^m f_i(\mu_{t_i})\right.+\\
&\hspace{4.5cm}\left.\cfrac{1}{2}\int_0^{t_m}\|v_s-b(.,\mu_s)\|_{L^2(\mu_s)}^2ds+I_0(\mu)\right].
\end{split}
\end{equation}
If we could exchange the supremum and the infimum and noticing that $$\sup_{f_1,...,f_m\in BL(\mesTdeux)}\sum_{i=1}^m (f_i(\nu_i)-f_i(\mu_{t_i}))$$ equals 0 if for all $1\le i\le m$ $\mu_{t_i}=\nu_i$ and $+\infty$ otherwise, we would have the result: 
\begin{align*}
I_{t_1,...,t_m}(\nu_1,...,\nu_m)&=\inf_{\substack{(\mu_s,v_s)_{s\in[0,t_m]}\\\mu_{t_i}=\nu_i, \,\forall 1\le i\le m\\\partial_s \mu_s+\nabla.(\mu_s v_s)=0}}\left[\cfrac{1}{2}\int_0^{t_m}\|v_s-b(.,\mu_s)\|_{L^2(\mu_s)}^2ds+I_0(\mu(0))\right]. 
\end{align*}
Nonetheless it does not seem clear how to exchange these supremum and infimum using a min-max theorem because we lack some convexity properties in $(\mu_s,v_s)_{s\in[0,t_m]}$. So, we prove directly the result. 
\newline
Let us define $$\mathcal J^m:=\inf_{\substack{(\mu_s,v_s)_{s\in[0,t_m]}\\\mu_{t_i}=\nu_i, \,\forall 1\le i\le m\\\partial_s \mu_s+\nabla.(\mu_s v_s)=0}}\left[\cfrac{1}{2}\int_0^{t_m}\|v_s-b(.,\mu_s)\|_{L^2(\mu_s)}^2ds+I_0(\mu(0))\right].$$
Using that we always have $\sup_{x\in A}\inf_{y\in B} f(x,y) \le \inf_{y\in B}\sup _{x\in A} f(x,y)$ for all sets $A,B$ and real-valued functions $f(x,y)$ we get 
$$I_{t_1,...,t_m}(\nu_1,...,\nu_m)\le \mathcal J^m.$$
So it remains to show the reverse inequality that is $\mathcal J^m\le I_{t_1,...,t_m}(\nu_1,...,\nu_m)$.
\newline
Of course the result is obvious if $I_{t_1,...,t_m}(\nu_1,...,\nu_m)=+\infty$ and so let us deal with the case $I_{t_1,...,t_m}(\nu_1,...,\nu_m)<+\infty$.
\newline
Recall from the proof of Proposition \ref{prop:largedevmarginals} that if we define the functions $f_i^k:\mesTdeux \to \R$ for $k\in\N$ and $1\le i\le m$ by $f_i^k(\mu)=-k W_2(\mu,\nu_i)$ which are in $BL(\mesTdeux)$ and $\Lambda_m$ is as in Proposition \ref{prop:largedevmarginals}, we get 
\begin{equation}
\label{eq:autreformulefonctiontaux}
\begin{split}
I_{t_1,...,t_m}(\nu_1,...,\nu_m)&=-\lim_{k\to+\infty}\Lambda_m(f_1^k,...,f_m^k)\\
&=-\lim_{k\to+\infty}\Lambda_0(\mathcal U^{f_1^k+\mathcal U^{f_2^k+...}(t_2-t_1)(\bullet)}(t_1,\bullet))\\
&=\lim_{k\to+\infty}\inf_{\substack{(\mu_s,v_s)_{s\in[0,t_m]}\\\partial_s \mu_s+\nabla.(\mu_s v_s)=0\\
\mu(0)=\mu}}\left[-\sum_{i=1}^m f_i^k(\mu_{t_i})\right.
+\\
&\hspace{4.5cm}\left.\cfrac{1}{2}\int_0^{t_m}\|v_s-b(.,\mu_s)\|_{L^2(\mu_s)}^2ds+I_0(\mu)\right],
\end{split}
\end{equation}
where for the last equality we used the same computations as in \eqref{eq:inversionI}.
So for $\varepsilon>0$, we can consider an integer $k_0$ such that for every $k\ge k_0$ \begin{align*}\inf_{\substack{(\mu_s,v_s)_{s\in[0,t_m]}\\\partial_s \mu_s+\nabla.(\mu_s v_s)=0\\
\mu(0)=\mu}}\left[-\sum_{i=1}^m f_i^k(\mu_{t_i})+\cfrac{1}{2}\int_0^{t_m}\|v_s-b(.,\mu_s)\|_{L^2(\mu_s)}^2ds+I_0(\mu)\right]& \le\\
&\hspace{-1cm}I_{t_1,...,t_m}(\nu_1,...,\nu_m)+\varepsilon.
\end{align*}
Hence we can construct a of functions $(\mu_s^k,v_s^k)_{s\in[0,t_m]}$ such that $\partial_s \mu_s^k+\nabla.(\mu_s^k v_s^k)=0$ in the sense of distributions that satisfies for every $k\ge k_0$
\begin{align*}
-\sum_{i=1}^m f_i^k(\mu_{t_i}^k)+\cfrac{1}{2}\int_0^{t_m}\|v_s^k-b(.,\mu_s^k)\|_{L^2(\mu_s^k)}^2ds+I_0(\mu_0^k) \le  I_{t_1,...,t_m}(\nu_1,...,\nu_m)+2\varepsilon.
\end{align*}
By definition of the functions $f_i^k$ this yields that for every $k\ge k_0$:
\begin{equation}
\label{eq:ineqboundepsilon}
k\sum_{i=1}^m W_2(\mu_{t_i}^k,\nu_i)+\cfrac{1}{2}\int_0^{t_m}\|v_s^k-b(.,\mu_s^k)\|_{L^2(\mu_s^k)}^2ds+I_0(\mu_0^k) \le  I_{t_1,...,t_m}(\nu_1,...,\nu_m)+2\varepsilon.
\end{equation}
First by noticing that we get from the previous inequality that for all $k\ge k_0$ $$k\sum_{i=1}^m W_2(\mu_{t_i}^k,\nu_i)\le I_{t_1,...,t_m}(\nu_1,...,\nu_m)+2\varepsilon,$$
dividing by $k$ and letting $k$ goes to $+\infty$ implies that for all $1\le i\le m$ \begin{equation}
\label{eq:W_2boundpourtouti}
W_2(\mu_{t_i}^k,\nu_i)\underset{k\to+\infty}{\longrightarrow}0.
\end{equation}
In the rest of the proof we denote by $C$ any constant that only depends on $t_1,...,t_m$, $\nu_1,...,\nu_m$, $\|b\|_\infty$ and $\varepsilon>0$.
\newline 
Let us remark that for $k\ge k_0$
\begin{align*}
\int_0^{t_m}\|v_s^k\|_{L^2(\mu_s^k)}^2ds&\le \cfrac{C}{2}\int_0^{t_m}\|v_s^k-b(.,\mu_s^k)\|_{L^2(\mu_s^k)}^2ds+C\int_0^{t_m}\|b(.,\mu_s^k)\|_{L^2(\mu_s^k)}^2ds\\
&\le  C(I_{t_1,...,t_m}(\nu_1,...,\nu_m)+2\varepsilon)+C=:C.
\end{align*}
Following the arguments of the proof of Theorem 5.14 in Section 5.3.2 of \cite{santambrogio2015opti}, from this bound we can classically construct $(\mu_s)_{s\in[0,t_m]}\in C([0,t_m],\mesTdeux)$ and $(v_s)_{s\in[0,t_m]}$ such that $$\int_0^{t_m}\|v_s\|_{L^2(\mu_s)}^2ds\le \liminf_{k\to+\infty} \int_0^{t_m}\|v_s^k\|_{L^2(\mu_s^k)}^2ds\le C$$
and so that $\partial_s\mu_s+\nabla\cdot(\mu_sv_s)=0$ in the sense of distributions and for which $(\mu^k)_{k\ge k_0}$ and $(E^k=\mu^k v^k)_{k\ge 0}$ (viewed as a sequence of measures) converge respectively uniformly toward $\mu$ and toward $E=\mu v$ in the weak sense up to extraction (we do not write the extraction to lighten notation). Let us note that by the superposition principle of Ambrosio-Gigli-Savaré \cite{AmbrosioGigliSavare2008} $\mu\in AC([0,t_m],\mesTdeux)$ where $AC$ denotes the space of absolutely continuous functions. Moreover, using \eqref{eq:W_2boundpourtouti}, we obtain that $\mu_{t_i}=\nu_i$ for all $1\le i\le m$.
\newline
Starting from \eqref{eq:ineqboundepsilon} we have for every $k\ge k_0$
\begin{align*}
\cfrac{1}{2}\int_0^{t_m}\|v_s^k-b(.,\mu_s^k)\|_{L^2(\mu_s^k)}^2ds+I_0(\mu_0^k)
\le  I_{t_1,...,t_m}(\nu_1,...,\nu_m)+2\varepsilon
\end{align*}
implying that for every $k\ge k_0$
\begin{align*}
\cfrac{1}{2}\int_0^{t_m}\|v_s^k\|_{L^2(\mu_s^k)}^2ds-\int_0^{t_m}\int_{\T^d}v_s^k(x)\cdot b(x,\mu_s^k)\mu_s(dx)ds+&\cfrac{1}{2}\int_0^{t_m}\|b(.,\mu_s^k)\|_{L^2(\mu_s^k)}^2ds+\\&\hspace{4cm}I_0(\mu_0^k)\\
\hspace{6cm}&\le  I_{t_1,...,t_m}(\nu_1,...,\nu_m)+2\varepsilon.
\end{align*}
Passing to the infimum limit and using that $I_0$ is lsc, that $b$ is Lipschitz and the convergence of $(\mu^k)_{k\ge 1}$ uniformly in $C([0,T],\mesTdeux)$ and the convergence of $(E^k)_{k\ge 1}$ toward $E$ we obtain that
 \begin{align*}
\cfrac{1}{2}\int_0^{t_m}\|v_s-b(.,\mu_s)\|_{L^2(\mu_s)}^2ds+I_0(\mu_0)\le  I_{t_1,...,t_m}(\nu_1,...,\nu_m)+2\varepsilon.
\end{align*}
In particular, by construction  we obtain that $$\mathcal J_m\le I_{t_1,...,t_m}(\nu_1,...,\nu_m)+2\varepsilon.$$
Let $\varepsilon$ goes to 0 to obtain the expected inequality. 
\end{proof}
Hence as a consequence of Lemma \ref{lemma:rewriteratefunctionmtuple} we obtain the following form for the rate function of the LDP of $(X_N)_{N\ge 1}$.
\begin{theorem}
\label{thm:largedev3}
Assume that the family $(\mathrm{ Law}(X_N(0)))_{N\ge 1}$ satisfies a large deviations principle in $\mesTdeux$ at speed $N/\sigma_N^2$ good rate function $I_0$.
\newline
Then the sequence of random variables $(X_N)_{N\ge1}$ satisfies a large deviations principle in $C([0,T],\mesTdeux)$ at speed $N/\sigma_N^2$ with good rate function $I$ where $$I(\mu)=I_0(\mu(0))+\cfrac{1}{2}\inf_{\substack{(v_s)_{s\in[0,T]};\\\partial_s \mu_s+\nabla.(\mu_s v_s)=0}}\int_0^T\int_{\T^d}|v_s(x)-b(x,\mu_s)|^2\mu_s(dx)ds$$ if $\mu:[0,T]\to \mesTdeux$ is absolutely continuous and $I(\mu)=+\infty$ otherwise.
\end{theorem}
\subsection{Extensions} The two main extensions that should be investigated are the case of a non-compact set instead of $\T^d$ and the case of a non-Lipschitz drift term $b$. 
\newline
Concerning a non-compact state space, the same strategy can be applied. The exponential tightness of the particle system can be obtained in the same way, provided that an additional assumption on the behavior of the drift at infinity is imposed. The main difference lies in the convergence of the transfer function to the solution of the limit PDE, and in the proof of the comparison principle for this PDE. In the non-compact case, to establish a comparison principle on the space of measures, one must first use Stegall's lemma to localize the problem and be able to consider a point of maximum before proceeding with the same proof as in the compact case (see Section 12.3 of \cite{bertucci2026analysis} or \cite{bertucci2026non}). Regarding the convergence, Ascoli's theorem will yield uniform convergence on compact subsets of the space of measures. It is then necessary in the proof of the convergence to add in the convergence proof a quadratic penalization term in the definition of test functions to ensure that the extremum points are attained. 
\newline
Extending this result to drift terms with lower regularity is a natural question, since it would allow one to handle singular interactions such as those arising in random matrix theory \cite{dyson1962,anderson2010}, or more generally in Riesz gases \cite{serfaty2024lectures}. This question shall be investigated in further work.

\subsection*{Acknowledgments} 
I thank Charles Bertucci (Ceremade) for the discussions on the paper and particularly on the analysis on the space of probability measures. I also thank Pierre-Louis Lions (Collège de France) for his courses at Collège de France which motivate this work.
\appendix
\section{Proof of Theorem \ref{thm:ldpmarginalstoproces}}
\label{Appendix:A}
\begin{lemma}
\label{lemma:largedevconttime1}
Let $(X_N)_{N\ge 1}$ be a sequence of random variables with values in $\mathbb Y$ that satisfies an LDP with good rate function $I$. If for any $m\ge 1$ and $0\le t_1 < t_2...< t_m \le T$, the sequence $(X_N(t_1),...,X_N(t_m))_{N\ge 1}$ satisfies an LDP in $\mathbb X^m$ with rate function $I_{t_1,...,t_m}$ we have $$I(f)\ge \sup_{m\ge 1}\sup_{0\le t_1<t_2...<t_m} I_{t_1,...,t_m}(f(t_1),...,f(t_m)).$$
\end{lemma}
\begin{proof}
It suffices to show that for every $f\in \mathbb Y$ and $0\le t_1 < t_2...< t_m \le T$ such that $(X_N(t_1),...,X_N(t_m))_{N\ge 1}$ satisfies an LDP in $\mathbb X^m$ with rate function $I_{t_1,...,t_m}$, then $$I(f)\ge I_{t_1,...,t_m}(f(t_1),...,f(t_m)).$$
We notice that for $\gamma>0$ $$B(f,\gamma)\subset \{g\in\mathbb Y\,|\, (g(t_1),...,g(t_m))\in B((f(t_1),...,f(t_m)),\gamma)\}.$$
\newline
Hence, using Remark \ref{remark:identiI} yields 
\begin{align*}
-I(f)&\le \underset{\gamma\to 0}{\lim}\underset{N\to +\infty}{\liminf}\cfrac{1}{\varepsilon_N}\log(\mP(X_N \in B(f,\gamma)))\\
&\le \underset{\gamma\to 0}{\lim}\underset{N\to +\infty}{\limsup}\cfrac{1}{\varepsilon_N}\log(\mP((X_N(t_1),...,X_N(t_m))\in \overline {B^m((f(t_1),...,f(t_m)),\gamma)}))\\
&=-I_{t_1,...,t_m}(f(t_1),...,f(t_m)).
\end{align*}
\end{proof}
The following lemma rigorously states that on a compact of $\mathbb Y$ the uniform convergence and the pointwise convergence are the same.
\begin{lemma}
\label{lemma:largedevconttime2}
Let $f\in\mathbb Y$. Fix $\gamma>0$ and $K$ a compact of $\mathbb Y$. There exists $0\le t_1<...<t_m\le T$ and $\gamma'>0$ such that 
$$K\cap \{g\in \mathbb Y\,|\, (g(t_1),...,g(t_m))\in B((f(t_1),...,f(t_m)),\gamma') \}\subset B(f,\gamma).$$
\end{lemma}
\begin{proof}
As a direct consequence of the Arzela-Ascoli theorem we can find a compact $k_\mathbb X$ of $\mathbb X$ and a modulus of continuity $h:\R^+\to\R^+$ that goes to 0 in 0 such that $$K\subset K_{k_\mathbb X,h}:=\{g\in \mathbb Y\,|\,g(t)\in k_\mathbb X, \, 0\le t\le T,\, \sup_{|t-s|\le \delta}d(f(t),f(s))\le h(\delta)\, \forall \delta>0 \}.$$
So it suffices to prove the result for $K$ of the form $K_{k_\mathbb X,h}$.
\newline
Let $\delta>0$ that shall be chosen later on. We define $t_{k+1}=k\delta$ for $0\le k\le \lfloor\frac{T}{\delta} \rfloor=:m-2$ and $t_m=T$. For every $g\in\mathbb Y$, we define on $[0,T]$ the function $P_g^\delta$ by  $P_g^\delta(t)=g(t_i)$ if $t\in[t_i,t_{i+1})$.
\newline
Using the triangular inequality we get 
$$\mathcal D(f,g)\le \mathcal D(f,P_f^\delta)+\mathcal D(P_f^\delta,P_g^\delta)+\mathcal D (P_g^\delta,g).$$
For $g\in K_{k_\mathbb X,h}\cap \{g\in \mathbb Y\,|\, (g(t_1),...,g(t_m))\in B^m((f(t_1),...,f(t_m)),\gamma' \}$, we have $$ \mathcal D(P_f^\delta,P_g^\delta)=\sup_{1\le i\le m} d(f(t_i),g(t_i))\le \gamma'$$ and since $|t_{i+1}-t_i|\le \delta$ we get
$$\mathcal D (P_g^\delta,g)=\sup_{1\le i\le m-1}\sup_{t\in[t_i,t_{i+1})}d(g(t),g(t_i))\le h(\delta).$$
Choosing $\delta>0$ small enough we have $$\mathcal D (P_g^\delta,g)\le h(\delta)\le \gamma'.$$
Moreover, since $f$ is continuous on a compact, it is uniform continuous, so there exists $\delta>0$ such that if $|t-s|\le \delta$ then $d(f(t),f(s))\le \gamma'$. 
We get for such small $\delta$ that 
$$\mathcal D(f,P_f^\delta)=\sup_{1\le i\le m-1}\sup_{t\in[t_i,t_{i+1})}d(f(t),f(t_i))\le \gamma'.$$
Hence if we summarize for a choice of $\delta$ small enough we have constructed $0\le t_1<...<t_m=T$ such that for all $g\in  K_{k_\mathbb X,h}\cap \{g\in \mathbb Y\,|\, (g(t_1),...,g(t_m))\in B^m((f(t_1),...,f(t_m)),\gamma' \}$ we get $\mathcal D(f,g)\le 3\gamma'$.
Choosing $\gamma'=\gamma/3$ gives the result.
\end{proof}
\begin{lemma}
\label{lemma:largedevconttime3}
Let $(X_N)_{N\ge 1}$ be a sequence of random variables with values in $\mathbb Y$ which satisfies an LDP with good rate function $I$ at speed $\varepsilon_N$. Assume that for all $0\le t_1<t_2...<t_m\le T$, the sequence $(X_N(t_1),...,X_N(t_m))_{N\ge 1}$ satisfies an LDP with good rate function $I_{t_1,...,t_m}$. Then, for every $f\in\mathbb Y$ we have $$I(f)\le \sup_{1\le i\le m}\sup_{0\le t_1<...<t_m  \le T} I_{t_1,...,t_m}(f(t_1),...,f(t_m)).$$
\end{lemma}
\begin{proof}
By Lemma \ref{lemma:largedevconttime2}, given $f\in\mathbb Y$, $\gamma>0$ and a compact $K$ that shall be chosen later on we can find $t_1<...<t_m$ and $\gamma'>0$ such that $$K\cap \{g\in \mathbb Y\,|\, (g(t_1),...,g(t_m))\in B^m((f(t_1),...,f(t_m)),\gamma') \}\subset B(f,\gamma).$$
We recall a fundamental lemma of Large deviations theory (see Lemma 1.2.15 of Section 1 of \cite{dembo2009large}) which is that for sequences of non negative real numbers $(a_N)_{N\ge 1}$ and $(b_N)_{N\ge 1}$ and a sequence $(c_N)_{N\ge 1}$ that goes to 0 we have $$\underset{N\to +\infty}{\limsup} (c_N\log(a_N+b_N))=\max\left(\underset{N\to +\infty}{\limsup}\,c_N\log(a_N),\underset{N\to +\infty}{\limsup}\,c_N\log(b_N)\right).$$
Using this property gives
\begin{equation}
\label{eq:largedevconttime}
\begin{split}
\max&\left(\underset{N\to +\infty}{\limsup}\cfrac{1}{\varepsilon_N}\log(\mP(X_N \in B(f,\gamma))),\underset{N\to +\infty}{\limsup}\cfrac{1}{\varepsilon_N}\log(\mP(X_N \notin K))\right)\\
&=\underset{N\to +\infty}{\limsup}\cfrac{1}{\varepsilon_N}\log(\mP(X_N \in B(f,\gamma))+\mP(X_N\notin K)) \\
&\ge \underset{N\to +\infty}{\limsup}\cfrac{1}{\varepsilon_N}\log(\mP(X_N \in B(f,\gamma)\cup K^c))\\
&\ge \underset{N\to +\infty}{\limsup}\cfrac{1}{\varepsilon_N}\log(\mP((X_N(t_1),...,X_N(t_m)) \in B((f(t_1),...f(t_m)),\gamma')))\\
&\ge\underset{N\to +\infty}{\liminf}\cfrac{1}{\varepsilon_N}\log(\mP((X_N(t_1),...,X_N(t_m)) \in B((f(t_1),...f(t_m)),\gamma')))\\
&\ge -\inf_{(y_1,...,y_m)\in B^m((f(t_1),...f(t_m)),\gamma')}I_{t_1,...,t_m}(y_1,...,y_m)\\
&\ge -I_{t_1,...,t_m}(f(t_1),...,f(t_m))\\
&\ge -\sup_{1\le i\le m}\sup_{0\le t_1<...<t_m  \le T} I_{t_1,...,t_m}(f(t_1),...,f(t_m)).
\end{split}
\end{equation}
Finally by exponential tightness for every $L>0$ there exists a compact $K_L$ of $\mathbb Y$ such that $$\underset{N\to +\infty}{\limsup}\cfrac{1}{\varepsilon_N}\log(\mP(X_N \notin K_L))\le -L. $$
Using these compacts in \eqref{eq:largedevconttime}, we get that for every $L>0$
\begin{align*}\max(\underset{N\to +\infty}{\limsup}\cfrac{1}{\varepsilon_N}\log(\mP(X_N \in B(f,\gamma))),-L&)\ge\\
& -\sup_{1\le i\le m}\sup_{0\le t_1<...<t_m  \le T} I_{t_1,...,t_m}(f(t_1),...,f(t_m)).
\end{align*} Letting first $L$ to $+\infty$ and then $\gamma$ to $0$ gives the result.
\end{proof}
We can now give a proof of Theorem \ref{thm:ldpmarginalstoproces}
\begin{proof}[Proof of Theorem \ref{thm:ldpmarginalstoproces}]
By Corollary \ref{coro:prokhorovlargedev}, since the sequence $(X_N)_{N\ge 1}$ is exponentially tight it suffices to prove the uniqueness of the good rate function $I$. From Lemma \ref{lemma:largedevconttime1} and Lemma \ref{lemma:largedevconttime3} the formula given in \eqref{eq:ratefunctionprocess} uniquely characterizes the good rate function $I$. 
\end{proof}
\section{Proof of Proposition \ref{prop:largedevmarginals}}
\label{Appendix:B}
\begin{proof}[Proof of Proposition \ref{prop:largedevmarginals}]
The formula \eqref{eq:varhandanfinidim} is a consequence of the Varadhan principle (Theorem \ref{thm:largedevvaradhan}).\newline
Now we shall prove that the knowledge of $\Lambda_m(f_1,...,f_m)$ for every $f_1,...,f_m\in BL(\mathbb X)$ uniquely characterize $I_m$.
\newline
Fix $(x_1,...,x_m)\in\mathbb X^m$. Choose for every $k\ge 1$ the functions $f_1^k,...,f_m^k\in BL(\mathbb X)$ defined by $f_i^k(y)=-k \min(d(x_i,y),1)$.
Let us prove that $$J(x_1,...,x_m):=\underset{k\to+\infty}{\lim}\inf_{y_1,...,y_m\in\mathbb X}\left[k\sum_{i=1}^m \min(d(x_i,y_i),1)+I_m(y_1,...,y_m)\right]=I_m(x_1,...,x_m).$$
We first deal with the case $I_m(x_1,...,x_m)<+\infty$.
Choosing the $x_1,...,x_m$ in the infimum immediately gives 
\begin{equation}
\label{eq:boundsimple}
\inf_{y_1,...,y_m\in\mathbb X}\left[k\sum_{i=1}^m \min(d(x_i,y_i),1)+I_m(y_1,...,y_m)\right]\le I_m(x_1,...,x_m).
\end{equation}
For the other inequality, let us consider a minimizing sequence of this infimum $y_1^k,...,y_m^k$ such that for every $k\ge 1$ we have 
\begin{equation}
\label{eq:minimiseur}
\begin{split}
k\sum_{i=1}^m \min(d(x_i,y_i^k),1)+I_m(y_1^k,...,y_m^k)&\le\\
&\hspace{-0.7cm}\inf_{y_1,...,y_m\in\mathbb X}\left[k\sum_{i=1}^m \min(d(x_i,y_i),1)+I_m(y_1,...,y_m)\right]+\cfrac{1}{k}\\
&\le I_m(x_1,...,x_m)+\frac{1}{k}.
\end{split}
\end{equation}
Since $I$ is non negative we get that for every $k\ge 1$
\begin{align*}
k\sum_{i=1}^m \min(d(x_i,y_i^k),1)\le I_m(x_1,...,x_m)+\frac{1}{k}.
\end{align*}
Dividing by $k$ and letting $k$ goes to $+\infty$ we get that for all $1\le i\le m$ $$y_i^k\underset{k\to+\infty}{\longrightarrow}x_i.$$
Again from \eqref{eq:minimiseur}, we obtain for every $k\ge 1$ 
$$I_m(y_1^k,...,y_m^k)-\cfrac{1}{k}\le \inf_{y_1,...,y_m\in\mathbb X}\left[k\sum_{i=1}^m \min(d(x_i,y_i),1)+I_m(y_1,...,y_m)\right].$$
Since $I_m$ is lsc, passing to the limit yields 
\begin{align*}
I_m(x_1,...,x_m)&\le\underset{k\to+\infty}{\liminf}\,I_m(y_1^k,...,y_m^k)-\cfrac{1}{k}
\\&\le \underset{k\to+\infty}{\liminf}\inf_{y_1,...,y_m\in\mathbb X}\left[k\sum_{i=1}^m \min(d(x_i,y_i),1)+I_m(y_1,...,y_m)\right].
\end{align*}
Combining this inequality with \eqref{eq:boundsimple} gives the result. 
In particular we have 
\begin{equation}
\label{eq:characIfinitedim}
\lim_{k\to+\infty}\Lambda_m(f_1^k,...,f_m^k)=-I_m(x_1,...,x_m).
\end{equation}
Now we deal with the case $I_m(x_1,...,x_m)=+\infty$. Assume by contradiction that $$\underset{k\to+\infty}{\liminf}\inf_{y_1,...,y_m\in\mathbb X}\left[k\sum_{i=1}^m \min(d(x_i,y_i),1)+I_m(y_1,...,y_m)\right]<+\infty.$$ We can find $C\in\R$ and a subsequence $(k_j)_{j\ge 1}$ that goes to $+\infty$ such that for all $j\ge 1$ $$\inf_{y_1,...,y_m\in\mathbb X}\left[k_j\sum_{i=1}^m \min(d(x_i,y_i),1)+I_m(y_1,...,y_m)\right]\le C.$$
By definition, we can find for all $j\ge 1$ $(y_1^j,...,y_m^j)\in\mathbb X^m$ that satisfies $$\left[k_j\sum_{i=1}^m \min(d(x_i,y_i^j),1)+I_m(y_1^j,...,y_m^j)\right]\le C+1.$$
We similarly get that for all $j\ge 1$ $$\sum_{i=1}^m \min(d(x_i,y_i^j),1)\le \cfrac{C+1}{k_j},$$
which gives that for all $1\le i\le m$, $y_i^j\underset{j\to+\infty}{\longrightarrow} x_i$. 
Again since $I_m$ is lsc we obtain $$+\infty=I_m(x_1,...,x_m)\le \liminf_{j\to+\infty}I_m(y_1^j,...,y_m^j)\le C+1,$$
which is a contradiction.
\newline
Hence in all cases we have 
\begin{equation}
\label{eq:characIfinitedim2}
\lim_{k\to+\infty}\Lambda_m(f_1^k,...,f_m^k)=-I_m(x_1,...,x_m).
\end{equation}
We now prove \eqref{eq:invformula}. Fix $(x_1,...,x_m)\in\mathbb X^m$. For $f_1,...,f_m\in BL(\mathbb X)$, we have $$\sum_{i=1}^m f_i(x_i)-I_m(x_1,...,x_m)\le \Lambda_m(f_1,...,f_m).$$ which gives that $$\sum_{i=1}^m f_i(x_i)-\Lambda_m(f_1,...,f_m)\le I_m(x_1,...,x_m).$$
Passing to the supremum gives $$\sup_{f_1,...,f_m \in BL(\mathbb X)}\left[\sum_{i=1}^m f_i(x_i)-\Lambda_m(f_1,...,f_m)\right]\le I_m(x_1,...,x_m).$$
Using the test functions $f_i^k(y)=-k\min(d(x_i,y),1)\in BL(\mathbb X)$ as before, for every $k\ge 1$, we get $$ \left[\sum_{i=1}^m f_i^k(x_i)-\Lambda_m(f_1^k,...,f_m^k)\right]\le \sup_{f_1,...,f_m \in BL(\mathbb X)}\left[\sum_{i=1}^m f_i(x_i)-\Lambda_m(f_1,...,f_m)\right].$$
Using that $f_i^k(x_i)=0$ for all $1\le i\le m$ and $k\ge 1$ and using $\eqref{eq:characIfinitedim2}$, passing to the limit yields 
$$ I_m(x_1,...,x_m)\le \sup_{f_1,...,f_m \in BL(\mathbb X)}\left[\sum_{i=1}^m f_i(x_i)-\Lambda_m(f_1,...,f_m)\right]$$
which concludes the proof. 

\end{proof}

\section{Proof of Proposition \ref{prop:compactness}}
\label{Appendix:D}
\begin{proof}
We shall apply Theorem \ref{thm:arzela-ascoli}.
In this proof, $C>0$ will be a constant that only depends on $f,b,T,\sup_N\sigma_N$ and $d$ that will be updated along the proof.
\newline 
First, by definition we have that for all $N\ge 1$, $$\sup_{t\in[0,T],\,x\in (\T^d)^N} |\Lambda_N^f(t,x)|\le \|f\|_{\infty}.$$  
It remains to show the equicontinuity property. We recall that for all $N\ge 1$, $\Lambda_N^f$ is the unique solution of \eqref{eq:edphofcole}.
\newline
The solution of \eqref{eq:edphofcole} can also be rewritten as the solution of a problem of optimal control
\begin{equation}
\label{eq:opticontropbfini}
\Lambda_N^f(t,x)=\sup_{\alpha\in\mathcal A, (Y_s)_{ s\le t}}\E_x\left[f(\mu_{Y_t^{\T^d}}^N)-\cfrac{1}{2N}\int_0^t\sum_{i=1}^N |\alpha_s^i|^2ds\right] ,
\end{equation}
where $\mathcal A$ is the set of control $\alpha :\Omega\times[0,T]\to (\R^d)^N$ that are adapted to the filtration $\sigma((Y_s)_{s\le t})$ and that are in $L^2(\Omega\times[0,T],(\R^d)^N)$ and $(Y_s^{\T^d})_{s\le t}$ is such that $(Y_s^{\T^d})_{s\le t}=(Y_s \text{ modulo } \Z^d)_{s\le t}$ where $(Y_s)_{s\le t}$ is the solution of the stochastic differential equations (in short SDE) $$dY_s^i=(b(Y_s^{i,\T^d},\mu_{Y_s^{\T^d}}^N)+\alpha_s^i)ds+\sigma_N dB_s^i, \, \forall 1\le i\le N,$$
with initial condition $Y_0^{\T^d}=x$.
It can be easily shown that the supremum is reached for the optimal control $$\alpha_s^{i,*}=N\nabla_{y_i}\Lambda_N^f(t-s,Y_s).$$
Let us mention that the formula \eqref{eq:opticontropbfini} can also be viewed as the so-called Boué-Dupuis formula \cite{boue1998variational}.
\newline
We now show the equicontinuity with respect to the space variable.
Fix $t\in[0,T]$ and $x,y\in (\T^d)^N$. Choosing the optimal control $\alpha^{t,x,*}$ associated to $(t,x)$ and using it for $(t,y)$ we have 
$$\Lambda_N^f(t,x)-\Lambda_N^f(t,y)\le \E\left[f\left(\mu_{Y_t^{\T^d,x}}^N\right)-f\left(\mu_{Y_t^{\T^d,y}}^N\right)\right],$$ where $(Y_s^{\T^d,x})_{s\le t}$ is the solution of the SDE associated to the optimal control $\alpha^{t,x,*}$ starting from $x$ and $(Y_s^{\T^d,y})_{s\le t}$ is the solution of the SDE with the same control but starting from $y$. 
Since $f$ is Lipschitz with constant $L_f$ we have 
\begin{equation}
\label{eq:ineqmodulecont}
\begin{split}\Lambda_N^f(t,x)-\Lambda_N^f(t,y)\le L_f\E (W_2(\mu_{Y_t^{\T^d,x}}^N,\mu_{Y_t^{\T^d,y}}^N))\le L_f\sqrt{\E (W_2(\mu_{Y_t^{\T^d,x}}^N,\mu_{Y_t^{\T^d,y}}^N)^2)}.
\end{split}
\end{equation}
We choose the label of the particles with respect to $x$ and $y$ such that $$\cfrac{1}{N}\,\sum_{k=1}^N|x_k-y_k|^2 =W_2(\mu_x^N,\mu_y^N)^2.$$
By definition of the distance $W_2$, we get \begin{equation}
\label{eq:bounW_2}
W_2(\mu_{Y_t^{\T^d,x}}^N,\mu_{Y_t^{\T^d,y}}^N)^2\le \frac{1}{N}\sum_{k=1}^N\left|Y_t^{k,y}-Y_t^{k,x}\right|^2=:S_t^{x,y,N}.
\end{equation}
Since $$Y_t^{k,y}-Y_t^{k,x}=y_k-x_k+\int_0^t \left[b(Y_u^{k,\T^d,y},\mu_{Y_u^{\T^d,y}}^N)-b(Y_u^{k,\T^d,x},\mu_{Y_u^{\T^d,x}}^N)\right]du$$ and using that $b$ is Lipschitz we get 
\begin{align*}
|Y_t^{k,y}-Y_t^{k,x}|^2&\le C|y_k-x_k|^2+C\left(\int_0^t|Y_u^{k,y}-Y_u^{k,x}|du\right)^2+C\left(\int_0^t W_2(\mu_{Y_u^{\T^d,y}}^N,\mu_{Y_u^{\T^d,x}}^N)du\right)^2\\
 &\le C|y_k-x_k|^2+C\left(\int_0^t|Y_u^{k,y}-Y_u^{k,x}|^2du+\int_0^t W_2(\mu_{Y_u^{\T^d,y}}^N,\mu_{Y_u^{\T^d,x}}^N)^2du\right),
\end{align*}
where we used the Cauchy-Schwarz inequality for the last bound and with $C$ a constant which depends on $d,L_b$ and $T>0$.
Summing for $1\le k\le N$ these identities and using \eqref{eq:bounW_2} yields $$S_t^{x,y,N}\le C\left(W_2(\mu_x^N,\mu_y^N)^2+2\int_0^t S_u^{x,y,N}du\right).$$
The Gronwall Lemma gives $$S_t^{x,y,N}\le C W_2(\mu_x^N,\mu_y^N)^2\exp(2Ct).$$
Thus, \eqref{eq:ineqmodulecont} implies $$\Lambda_N^f(t,x)-\Lambda_N^f(t,y)\le C W_2(\mu_x^N,\mu_y^N)\exp(Ct),$$
where $C$ is a constant that only depends on $d,L_b,T$ and $L_f$.
\newline
Exchanging the role of $x$ and $y$ gives the existence of a constant $C>0$ that depends on $d,L_b,T$ and $L_f$ such that we have the following uniform bound
\begin{equation}
\label{eq:equicontspace}
\sup_{t\in [0,T]}\left|\Lambda_N^f(t,x)-\Lambda_N^f(t,y)\right|\le C W_2(\mu_x^N,\mu_y^N).
\end{equation}
We now show the equicontinuity with respect to the time variable.
Using the Dynamic programming principle, for every $h\ge 0$, 
\begin{equation}
\label{eq:dpp}
\Lambda_N^f(t+h,x)=\sup_{\alpha\in\mathcal A, (Y_s)_{ s\le h}}\E_x\left[\Lambda_N^f(t,Y_h^{\T^d})-\cfrac{1}{2N}\int_0^h\sum_{i=1}^N |\alpha_s^i|^2ds\right],
\end{equation}
where as before $\mathcal A$ is the set of control $\alpha :\Omega\times[0,h]\to (\R^d)^N$ that are adapted to the filtration $\sigma((Y_s)_{s\le h})$ and that are in $L^2(\Omega\times[0,h],(\R^d)^N)$ and $(Y_s)_{s\le h}$ is the solution of the SDE $$dY_s^i=(b(Y_s^{i,\T^d}\mu_{Y_s^{\T^d}}^N)+\alpha_s^i)ds+\sigma_N dB_s^i, \, \forall 1\le i\le N.$$
Choosing the control $\alpha=0$ we have $$\Lambda_N^f(t+h,x)-\Lambda_N^f(t,x)\ge \E_x[\Lambda_N^f(t,Y_h^{\T^d})- \Lambda_N^f(t,x)],$$ where $$dY_r^i=b(Y_r^{i,\T^d},\mu_{Y_r^{\T^d}}^N)dr+\sigma_N dB_r^i, \, \forall 1\le i\le N.$$
Using the equicontinuity with respect to the space variable, we get 
\begin{align*}
\Lambda_N^f(t+h,x)-\Lambda_N^f(t,x)&\ge -C \E_x(W_2(\mu_{Y_h^{\T^d}}^N,\mu_x^N))\\
&\ge -C \sqrt{\E_x\left[W_2(\mu_{Y_h^{\T^d}}^N,\mu_x^N)^2\right]}\\
&\ge -C \sqrt{\E_x\left[\cfrac{1}{N}\sum_{i=1}^N\left|Y_{h}^i-x_i\right|^2\right]} .
\end{align*}
Applying the Itô formula gives
$$ \E_x[|Y_h^i-x_i|^2]\le C \|b\|_\infty^2 h^2+C\sigma_N^2 h.$$
Thus we obtain $$\Lambda_N^f(t+h,x)-\Lambda_N^f(t,x)\ge -C (h+\sqrt{h}),$$
where $C$ is a new constant that depends on the same parameters as before and $\sup_{N\ge 1} \sigma_N$.
\newline
For the upper bound, we start by 
\begin{equation}
\label{eq:upboundequipilibtime}
\Lambda_N^f(t+h,x)-\Lambda_N^f(t,x)=\sup_{\alpha\in\mathcal A, (Y_s)_{ s\le h}}\E_x\left[\Lambda_N^f(t,Y_h^{\T^d})-\Lambda_N^f(t,x)-\cfrac{1}{2N}\int_0^h\sum_{i=1}^N |\alpha_s^i|^2ds\right].
\end{equation}
Again by the equicontinuity with respect to the space variable we get $$\E_x[\Lambda_N^f(t,Y_h^{\T^d})-\Lambda_N^f(t,x)]\le C\E_x\left[W_2(\mu_{Y_h^{\T^d}}^N,\mu_x^N)\right]\le C\sqrt{\cfrac{1}{N}\sum_{i=1}^N \E_x\left[|Y_h^i-x_i|^2\right]}.$$
By the Itô formula and the Cauchy-Schwarz inequality we have 
\begin{align*}
\E_x[|Y_h^i-x_i|^2] &\le C h^2+h\E_x\left[\int_0^h|\alpha_s^i|^2ds\right]+C h\\
&\le C(h^2+h)+h\,\E_x\left[\int_0^h |\alpha_s^i|^2 ds\right].
\end{align*}
Combining this inequality with \eqref{eq:upboundequipilibtime} we obtain
\begin{equation}
\label{eq:boundequicontime}
\begin{split}
\Lambda&_N^f(t+h,x)-\Lambda_N^f(t,x)\le \\
&\sup_{\alpha\in\mathcal A, (Y_s)_{ s\le h}}\left(C\sqrt{C(h^2+h)+h\,\E_x\left[\cfrac{1}{N}\int_0^h\sum_{i=1}^N|\alpha_s^i|^2 ds\right]}-\cfrac{1}{2}\,\E_x\left[\cfrac{1}{N}\int_0^h \sum_{i=1}^N|\alpha_s^i|^2ds\right]\right).
\end{split}
\end{equation}
Let $\beta:=\E_x\left[\cfrac{1}{N}\dis\int_0^h \sum_{i=1}^N|\alpha_s^i|^2ds\right]$ and remark that we can bound the previous term as follows 
\begin{align*}
C\sqrt{C(h^2+h)+h\,\beta}-\cfrac{1}{2}\,\beta&\le C\sqrt{C(h^2+h)}+C\sqrt{h\beta}-\cfrac{1}{2}\,\beta\\
&\le C\sqrt{C(h^2+h)}+\cfrac{C^2}{2}\,h+\cfrac{1}{2}\,\beta-\cfrac{1}{2}\,\beta\\
&\le C\sqrt{C(h^2+h)}+\cfrac{C^2}{2}\,h\\
&\le C(h+\sqrt{h}),
\end{align*}
for a constant $C>0$ as before.
Using this bound with \eqref{eq:boundequicontime} yields
\begin{align*}
\Lambda&_N^f(t+h,x)-\Lambda_N^f(t,x)\le C(h+\sqrt{h}),
\end{align*}
where $C$ is a constant which depends on $d,b,f,T,\sup_{N\ge 1}\sigma_N$.
\newline
Hence we obtain the existence of a modulus of continuity $\omega$ such that for every $x\in(\T^d)^N$ and $h\ge 0$, we have $$|\Lambda_N^f(t+h,x)-\Lambda_N^f(t,x)|\le \omega(h).$$
This concludes the proof about the equicontinuity of the family $(\Lambda_N^f)_{N\ge 1}$.
\end{proof}

\section{Proof of Proposition \ref{coro:liftsolimpliqueviscosol}}
\label{Appendix:C}
As expected the differential of the lift of functions defined on $\mesTdeux$ is linked with the differential on $\mesTdeux$ as introduced in Definition \ref{def:differentialtore}.
\begin{lemma}
\label{prop:subhilbandmesure}
Let $\mathcal U:\mesTdeux\to\R$ be a usc function and define its lift $\tilde U: \mathcal H\to \R$ by $\tilde U(X)=\mathcal U(\mathrm{ Law}(X))$. Then, for all $\mu\in \mesTdeux$ and $\psi\in \partial^+\mathcal U(\mu)$, for any $(X,Z)\in L^2((0,1),\T^d)\times L^2((0,1),\R^d)$ such that $\mathrm{Law}(X,Z)=\mu(dx)\psi_x(dz)$ we have $Z\in \partial^+\tilde U (X)$.
\end{lemma}
Let us mention that we can conversely construct elements in the sub-differential of the initial function from elements in the sub-differential of its lift (see Proposition 5.21 of \cite{bertucci2026analysis}) but in what follows we will only use the result stated in Proposition \ref{prop:subhilbandmesure}. 
\begin{proof}
Let $V\in L^2((0,1),\R^d)$ we can define the plans $\eta\in\mathcal P(\T^d\times\R^d)$ which is the law of $(X,V)$ and $\Gamma\in \mathcal P(\T^d\times\R^d\times\R^d)$ which the law of $(X,Z,V)$. 
Noticing that by definition the law of $X+V\text{ modulo }\Z^d$ is $\exp_\sharp\eta$ and applying the fact that $\psi\in \partial^+ U(\mu)$ yields
\begin{align*}
\tilde U(X+V\text{ modulo }\Z^d)&=\mathcal U(\exp_\sharp \eta)\\
&\le \mathcal U(\mu)+\int_{\T^d\times\R^d\times\R^d}z\cdot v\,\Gamma(dx,dz,dv)+C_{\T^d}(\eta)\omega(C_{\T^d})\\
&=\tilde U(X)+\E[X\cdot Z]+\|V\|_{L^2}\omega(\|V\|_{L^2}),
\end{align*}
where $\omega$ is a modulus of continuity associated with the sub-differential at $\mu$ of $\mathcal U$. This gives the result.
\end{proof}
We can now prove Proposition \ref{coro:liftsolimpliqueviscosol}
\begin{proof}[Proof of Proposition \ref{coro:liftsolimpliqueviscosol}]
We only proof the fact that it is a sub-solution, the super-solution property can be done similarly.
\newline 
Let $(t,\mu)\in (0,T]\times \mesTdeux$ and $(\theta,\psi)\in \partial^+\mathcal U(t,\mu)$ and take $$(X,Z)\in
 L^2((0,1),\T^d)\times L^2((0,1),\R^d)$$ such that $\mathrm{Law}(X,Z)=\mu(dx)\psi_x(dz)$. Since by Proposition \ref{prop:subhilbandmesure}, $(\theta,Z)\in\partial^+\tilde U(t,X)$ we get 
\begin{align*}
\theta-\E[b(X,\mathrm{Law}(X))\cdot Z]-\frac{1}{2}\E[\,|Z|^2]\le 0.
\end{align*}
Using that $\mathrm{Law}(X,Z)=\mu(dx)\psi_x(dz)$ yields 
\begin{align*}
\theta-\int_{\T^d\times\R^d} b(x,\mu)\cdot z\,\psi_x(dz)\mu(dx)-\frac{1}{2}\int_{\T^d\times\R^d}|z|^2 \,\psi_x(dz)\mu(dx)\le 0.
\end{align*}
This concludes the proof.
\end{proof}
\renewcommand{\MR}[1]{}
\bibliographystyle{smfplain}
\bibliography{ref}

\end{document}